\documentclass[twocolumn]{autart}    % Enable this line and disable the 
\usepackage{amsmath,amssymb}
\usepackage{color}
\usepackage{cite}
\usepackage{natbib}
\usepackage{numcompress}
\input{mysymbol.sty}
\usepackage{needspace}

\usepackage{graphicx}
\usepackage{nccmath}% Include this line if your 
\setcitestyle{round}
\newenvironment{proof}{{\it proof. ~~}}{  \hfill \\ \\}
\newtheorem{assumption}{Assumption}
\newtheorem{remark}{Remark}

\newtheorem{theorem}{Theorem}
\newtheorem{lemma}{Lemma}

\newtheorem{corollary}{\bf Corollary}

\newcommand{\rhop}{{\mathbf p}}
\newcommand{\X}{{\mathcal X}}
\newcommand{\n}{{\mathbf n}}
\newcommand{\F}{{\mathcal F}}

\newcommand{\R}{{\mathbb R}}

\newcommand{\fz}{{f_0}}

\newcommand{\Oc}{{\mathcal O}}
\newcommand{\Cc}{{\mathcal C}}

\newcommand{\black}[1]{{\color{black}#1}}
\renewcommand{\blue}[1]{{\color{black}#1}}

\newcommand{\starx}{{(x-x^*)}}

\begin{document}

\begin{frontmatter}

\title{Navigation of a Quadratic Potential with\\ Ellipsoidal Obstacles} % Title, preferably not more 

\thanks[footnoteinfo]{Department of Electrical and Systems Engineering at the University of Pennsylvania}

\author{Harshat Kumar}\ead{harshat@seas.upenn.edu},    % Add the 
\author{Santiago Paternain}\ead{spater@seas.upenn.edu},               % e-mail address 
\author{Alejandro Ribeiro}\ead{aribeiro@seas.upenn.edu}  % (ead) as shown

%\address[penn]{Upenn}  % Please supply                                              

\begin{keyword}                           % Five to ten keywords,  
Guidance Navigation and Control; Path Planning; Obstacle Avoidance; Local Control               % chosen from the IFAC 
\end{keyword}                             % keyword list or with the 
                                          % help of the Automatica 
                                          % keyword wizard

\begin{abstract}                          % Abstract of not more than
	 Given a convex quadratic potential of which its minimum is the agent's goal and a \blue{Euclidean} space populated with ellipsoidal obstacles, one can construct a Rimon-Koditschek (RK) artificial potential to navigate. Its negative gradient attracts the agent toward the goal and repels the agent away from the boundary of the obstacles. This is a popular approach to navigation problems since it can be implemented with \blue{local spatial} information that is acquired during operation time. However, navigation is only successful in situations where the obstacles are not too eccentric (flat). This paper proposes a modification to gradient dynamics that allows successful navigation of an environment with a quadratic cost and ellipsoidal obstacles regardless of their eccentricity. This is accomplished by altering gradient dynamics with a Hessian correction that is intended to imitate worlds with spherical obstacles in which RK potentials are known to work. The resulting dynamics simplify by the quadratic form of the obstacles. Convergence to the goal and obstacle avoidance is established from \blue{almost} every initial position \blue{(up to a set of measure one)} in the free space, \blue{with mild conditions on the location of the target}. Results are \blue{corroborated empirically with} numerical \blue{simulations}.
\end{abstract}

\end{frontmatter}

%!TEX root = autosam.tex

\section{Introduction} \label{sec_intro}

Path planning, sometimes formulated as reaching the minimum of a potential function from a start configuration while avoiding collisions with obstacles, is a cornerstone problem in controls and robotics \citep{lavalle2006planning, bhattacharya2007motion, murphy2008search}. In this paper we develop a navigation function approach that is guaranteed to reach the minimum of an arbitrary quadratic convex potential in a space with an arbitrary number of ellipsoidal obstacles of arbitrary eccentricity. 

To better explain this contribution it is important to emphasize that navigation function approaches to path planning occupy an appealing middle ground in terms of complexity and quality of trajectories \citep{rimon1992exact, loizou2017navigation, vrohidis2018prescribed, tanner2005formation, ghaffarkhah2009communication, paternain2016stochastic}. In one extreme, bug algorithms follow potential gradients until they hit border obstacles, at which point they follow the border until the projection of the direction to the destination on the obstacle's tangent plane pushes it away \citep{taylor2009bug}. Bug algorithms rely on simple \emph{local} sensing of gradients and obstacles and are guaranteed to reach the target destination. But they may do so by following excessively long trajectories. In the other extreme, minimum path length search algorithms such as A$^*$ \citep{hart1968formal} and random trees \citep{lavalle1998rapidly} build graphs that describe the geometry of the environment and find trajectories of optimal length. But to do so they require access to complex \emph{global} sensing of the environment.

Navigation function approaches combine the goal potential with repulsive potentials that push the agent away from the obstacles. This implies they can still be implemented with \emph{local} sensing of gradients and obstacles while empirical evidence shows they find trajectories to the goal that are better than obstacle following bug algorithms \citep{loizou2017navigation, vrohidis2018prescribed, tanner2005formation, ghaffarkhah2009communication, paternain2016stochastic}. The cost to pay for this appealing tradeoff between sensing complexity and trajectory length is the possibility of failure. Locally implementable navigation functions are guaranteed to reach the goal, but they are difficult to construct except for conservative geometries. Famously, they are known to work always for spherical potentials in worlds with spherical obstacles \citep{koditschek1990robot}. \black{With the addition of an analytic switch, local diffeomorphisms can be constructed to enable navigation of a single point agent in star worlds \citep{rimon1992exact}. Since then, there has been a significant effort to implement navigation functions locally \emph{without} a diffeomorphism \citep{paternain2018navigation, tanner2005formation, loizou2012navigation}. Such constructions extend the applications to to multiple agents \citep{dimarogonas2008decentralized}, non-point agents \citep{arslan2016sensor}, and minimal adjustment of a single tuning parameter \citep{lionis2008towards}. }  \emph{All} of these efforts, however, require that the obstacles are sufficiently curved \citep{filippidis2012navigation}. 
In the specific case of ellipsoidal obstacles, this puts a limit on their eccentricity. %This fact introduces an inherent risk in the implementation of the \black{Rimon-Koditscheck navigation function construction}, as it is unrealistic to assume a prioiri that eccentricity conditions are met.

%This risk has motivated the development of alternative approaches inspired by navigation functions. For instance, arbitrary star worlds -- which include arbitrary convex obstacles as a particular case -- can be mapped into spherical worlds \citep{rimon1991construction,loizou2011navigation, loizou2012navigation}. However, this mapping requires global information, thereby negating the advantage of being able to navigate with local information only. \black{Local mappings have been considered \cite{rimon1992exact}}
\blue{More recently, \cite{vasilopoulos2020reactive} have shown a method to construct a diffeomporphism via polygonal decomposition on-the-fly with non-convex obstacles using real-time perception; however, they require the shape of the obstacles to be "familiar" or come from a known obstacle class. Indeed, mapped model space after applying the diffeomporphism does need to satisfy the sufficiently curved condition.}

In this work, we consider ellipsoidal worlds with \emph{arbitrary} eccentricity. By drawing connections to second order optimization \citep[Ch. 9]{boyd2004convex}, we \black{directly propose dynamics given by a Hessian correction on the locally implementable navigation function of \cite{paternain2018navigation}. The correction on the quadratic obstacles simplifies so that the agent only requires easily obtainable information; namely, the agent must estimate the direction to the goal and obstacle centers as well as the distances to the obstacles and goal. Implementable in practice \citep{li2004least,fitzgibbon1999direct}, we show that with these four quantities and \blue{with mild conditions on the position of the target}, the agent is guaranteed to reach the target \blue{from almost all starting conditions} while avoiding obstacles along the way.} % This allows the convergence guarantees to be extended to the class of ellipsoidal obstacles with arbitrary eccentricity, \black{a condition that has \emph{not} been achieved by the aforementioned literature}. Implementable in practice \citep{li2004least,fitzgibbon1999direct}, ellipsoidal fitting of obstacles results in a tighter fit of of obstacles compared to their conservative spherical counterparts. The Hessian correction on the gradient based dynamics simplifies the dynamics so that the agent only requires easily obtainable information. In particular, the agent must estimate the direction to the goal and obstacle centers as well as the distances to the obstacle and goal. 

%In this work, we show that the proposed method functions \textit{without} second order information while still enjoying the same convergence guarantees in ellipsoidal worlds.  As such, this work extends the merits of navigation function approaches to \black{ellipsoidal obstacles; in particular } where locally implementable navigation functions are difficult to construct. 

The paper is organized as follows. In section \ref{problem_form}, we formally introduce the path planning problem and the navigation function approach. In section \ref{sec_correction}, we present our Hessian-corrected and simplified dynamics. Following the proof of the main result in section \ref{proof_of_theorem}, numerical results are presented in section \ref{num_results} which showcases the success in spaces where the traditional navigation function approaches fail. \black{We conclude in section \ref{conclusion} with a summary and possible extensions of this work.}

\section{Potential, Obstacles, \& Navigation Functions} \label{problem_form}

We consider the problem of a point agent navigating a quadratic potential in a space with ellipsoidal punctures. Formally, let $\mathcal{X} \subset \R^n$ be a non empty compact convex \blue{domain} that we call the workspace, and let $\fz: \X \to \R_+$ be a convex strictly quadratic function that we call the potential. A point agent is interested in reaching the target destination $x^* \in \ccalX$ which is defined as the minimizer of the potential. \blue{Without loss of generality, we }use the standard squared Euclidean distance to the target
\begin{equation}\label{eqn_goal_def}
x^* = \argmin_{x \in \X} \fz(x)
    := \frac{1}{2} \|x-x^*\|^2.
\end{equation}
In some navigation problems, arbitrary quadratic functions are of interest \citep{paternain2018navigation}. For future reference, we denote the minimum and maximum eigenvalues of $Q:= \nabla^2 \fz (x)$ as $0<\lambda_\textrm{min}\leq \lambda_\textrm{max}$.

The workspace $\ccalX$ is populated by $m$ ellipsoidal obstacles $\Oc_i \subset \X$ for $i = 1, \dots, m$ which are closed and have a non empty interior. \blue{We define the free space as the complement of the obstacle set relative to the workspace, 
\begin{equation} \label{eqn:free_space}
   \F := \X \ \backslash\  \Big(\, {\textstyle \bigcup\limits_{i = 1}^m}\Oc_i \, \Big),
\end{equation} 
We assume that each obstacle, or each connected component of the complement of the workspace is an ellipsoid. Formally, we have the assumption 
\begin{assumption} \label{as:1} {\bf Obstacles are ellipsoids.} 
Each obstacle is represented as the \blue{zero} sublevel set of a proper convex quadratic function $\beta_i: \R^n \to \R$
\begin{equation} \label{equ:beta1_def}
   \beta_i(x) = \frac{1}{2}(x - x_i)^\top A_i (x - x_i) - \frac{1}{2}r_i^2 ,
\end{equation}
where $A_i$ is a positive definite matrix with minimum and maximum eigenvalues $0 < \mu^i_\textrm{min} \leq \mu^i_\textrm{max}$, $x_i$ is the ellipsoid center and $r_i$ is the maximum axis length so that 
\begin{equation} \label{equ:obs_def_beta}
   \Oc_i = \left\{x \in \X \given \beta_i(x) \leq 0 \right\}.
\end{equation}
\end{assumption}
}
From \eqref{equ:beta1_def} and \eqref{equ:obs_def_beta} it follows that $\Oc_i$ is an ellipsoid centered at $x_i$ with axes given by the eigenvectors of $A_i$. The length of the axis along the $k$th eigenvector is \blue{$r_i\mu^i_k $}. In particular, the length of the minor axis is $r_i\mu^i_\textrm{min}$ and the length of the major axis is $r_i\mu^i_\textrm{max}$.

While it is true that convergence guarantees have been given for more complex obstacles, such as tori, cylinders, one-sheet hyperboloids, and convex obstacles in general \cite{filippidis2012navigation, paternain2018navigation}, these obstacles must be \emph{sufficiently curved}. In the case of ellipsoidal obstacles, this puts a restriction on their eccentricity. Focusing on the ellipsoidal case, this work guarantees convergence on ellipsoidal worlds with \emph{arbitrary} eccentricity. 

We further introduce a concave quadratic function $\beta_0: \R^n \to \R$ so that we write the workspace as a superlevel set,
\begin{equation}\label{eqn_workspace}
\X = \left\{x \in \R^n \given \beta_0(x) \geq 0 \right\}.
\end{equation}
\blue{The fact that the workspace is bounded admits the following bounds which will appear in our convergence analysis. Let $B$ be a strictly positive constant. For any four points $a, b, c, d \in \X$, the absolute value of the inner product of the distances $a-b$ and $c-d$ is bounded by
\begin{equation} \label{equ:distance_bound}
    |(a-b)^\top (c-d)| \leq B.
\end{equation}
Additionally, let
\begin{equation} \label{equ:betai_fz_bound}
\Lambda_0 := \max_{i, x\in  \X} \beta_i(x),~ \textrm{and}~ P_0 := \max_{x\in  \X} \fz(x).
\end{equation}
}
%The latter is possible to do since superlevelsets of concave functions are convex sets \citep{boyd2004convex}.

The navigation problem we want to solve is one in which the agent stays in the interior of the workspace at all times, does not collide with any obstacle, and approaches the goal at least asymptotically. 

For a formal specification we specify the agent's goal as that of finding a trajectory $x(t)$ such that \blue{for all $t \geq 0$},
\begin{equation}\label{eqn_problem_interest}
   x(t) \in \F,  
        ~\textrm{and} 
              ~\lim_{t \to \infty} x(t) = x^*.
\end{equation}
The problem is feasible when $x^*, x(0) \in \mathcal{F}$. 

\subsection{Navigation Functions}

A navigation function is a twice continuously differentiable function defined on the free space that satisfies three properties: (i) It has a unique minimum at $x^*$. (ii) All of its critical points are are nondegenerate. (iii) Its maximum is attained at \black{every point on} the boundary of the free space. These three properties guarantee that if an agent follows the negative gradient of the navigation function, it will converge to the minimum of the navigation function without running into the boundary of free space for almost every initial condition \citep{koditschek1990robot}. Thus, it is possible to recast \eqref{eqn_problem_interest} as the problem of finding a navigation function whose minimum is at the goal destination $x^*$. This is always possible to do since for any manifold with boundary it is guaranteed that \blue{such a} function exists \citep{{rimon1992exact}}. \black{In practice}, depending on the geometry of the freespace the navigation functions are constructed differently. For instance, in \blue{sphere} worlds, Rimon-Koditschek artificial potentials can be used \citep{rimon1992exact}, and in topologically complex ones navigation functions based on harmonic potentials are preferred \citep{loizou2011navigation, loizou2012navigation}. The family of Rimon-Koditshek potentials \blue{was} extended to \blue{enable the navigation of} convex potentials in a space of convex obstacles \citep{filippidis2012navigation, paternain2018navigation, rimon1991construction}. However, some geometric conditions restrict its application directly, \black{which we will now elucidate}. %\red{Navigation Functions are guaranteed to exists on any smooth manifold with boundary \citep{koditschek1990robot}. However, their shape is highly dependent on the geometry of the free space. A family of functions that has been widely studied for focally admissible obstacles \citep{filippidis2012navigation, paternain2018navigation}-- and later extended to case of star worlds \citep{rimon1991construction}-- is that of Rimon-Kodischek potentials.} \black{This is too much detail. You are not doing a literature review. You just need to mention things that are necessary for your presentation.} \green{A family of navigation functions that can be constructed to navigate a convex potential such as $f_0$ in a space of nonintersecting convex obstacles such as the $\Oc_i$ we consider here is the Rimon-Koditschek potential \citep{filippidis2012navigation, paternain2018navigation, rimon1991construction}.} 

%To explain the construction of these potentials we use the definition of the obstacles and the workspace provided in \eqref{equ:obs_def_beta} and \eqref{eqn_workspace} we can write an analytical expression for free space. To that end, 
%By Assumption \ref{as:1}, only the function $\beta_i$ can be negative inside of obstacle $\Oc_i$. It follows that $\beta(x)$ is negative if and only if the argument $x$ is inside of some obstacle. 
%From this, we define the free space as the set of points where the product of $\beta(x)$ and $\beta_0(x)$ is positive,
%
%\begin{equation}\label{eqn_new_freespace}
%   \F = \left\{ x\in \R^n \given \beta_0(x) \beta(x) > 0 \right\}.
%\end{equation}
%
%We can now think of the product $\beta_0(x) \beta(x)$ as a barrier function that we must guarantee stays positive during navigation.
The Rimon-Koditschek \blue{navigation function} \black{is formally defined} by $\varphi_k: \F \to \R_+$ with parameter $k\in \mathbb{R}$ as
\begin{equation} \label{equ:nav_fun}
   \varphi_k(x)  = \frac{\fz(x)}
                        { \left(\fz^k(x) + \beta_0(x)\beta(x) \right)^{1/k}},
\end{equation}
where the function $\beta:\R^n \to \R$ is the product of all the obstacle equations,
\begin{equation} \label{beta_prod}
   \beta(x) = \prod_{i = 1}^m \beta_i(x).
\end{equation}
It was established that $\varphi_k(x)$ is a navigation function when $k$ is sufficiently large under some restrictions on the shape of the obstacles, the potential function, and position of the goal \citep{filippidis2012navigation, paternain2018navigation}.

\blue{Namely,} $\varphi_k(x)$ in \eqref{equ:nav_fun} is not always a valid navigation function because for some geometries it can have several local minima as critical points \blue{for all $k>0$}. For the case of a quadratic potential and ellipsoidal obstacles that we consider here, a sufficient condition for $\varphi_k(x)$ to be a valid potential is when \citep[Theorem 3]{paternain2018navigation}
\begin{equation} \label{equ:condition}
   \frac{\lambda_\textrm{max}}{\lambda_\textrm{min}} \times 
   \frac{\mu^i  _\textrm{max}}{\mu^i  _\textrm{min}} 
         < 1 + \frac{d_i}{r_i\mu^i_\textrm{max}},
\end{equation}
where $d_i = \|x_i - x^* \|$ is the distance from the center of the ellipsoid to the goal. When \eqref{equ:condition} fails, $\varphi_k$ might fail to be a navigation function because it may present a local minimum on the side of the obstacle opposite the target.
%\red{To satisfy the inequality we need the curvature of the obstacles to be thin relative to the curvature of the potential. This is what happens on the example in the left in Figure \ref{fig:prob_form} where the level sets of the potential and the obstacle are circular. The inequality is violated when the curvature of the obstacle is too wide with respect to the objective function. This is what happens on the example in the right of Figure \ref{fig:prob_form} where the ellipsoidal obstacle is wide and the level sets of the objective are narrow. This geometry generates a local minima in $\varphi_k(x)$ irrespectively of the value of $k$.}

\blue{An} important consequence \blue{of} \eqref{equ:condition} is that the artificial potential $\varphi_k(x)$ in \eqref{equ:nav_fun} may fail to solve the navigation problem specified in \eqref{eqn_problem_interest}. Indeed, it will fail whenever the obstacles are wide with respect to the potential level sets. On the other hand, notice that when the attractive potential is rotationally symmetric and the obstacles are spherical, the left hand side of the previous expression is equal to one, and thus the condition is always satisfied. The main contribution of this paper is to leverage this observation \blue{as a motivation for introducing a correction to the gradient field of an Rimon-Koditschek navigation function that results in a field construction that is valid in all environments with ellipsoidal obstacles, and for any quadratic potential. } %to introduce a correction in the gradient field arising from a navigation function that allows to solve \eqref{eqn_problem_interest} in {\it all} environments with quadratic potentials and ellipsoidal obstacles.%\red{(REMOVE?) \emph{without} using second order information about the potential nor the obstacles.}

%%%%%%%%%%%%%%%%%%%%%%%%%%%%%%%%%%%%%%%%%%%%%%%%%%%%%%%%%%%%%%%%%%%%%%%%%%%%%%%%
%%%   F   I   G   U   R   E   %%%%%%%%%%%%%%%%%%%%%%%%%%%%%%%%%%%%%%%%%%%%%%%%%%
%%%%%%%%%%%%%%%%%%%%%%%%%%%%%%%%%%%%%%%%%%%%%%%%%%%%%%%%%%%%%%%%%%%%%%%%%%%%%%%%
%

%

%\begin{figure}[!t] 
%\centering
%\includegraphics[width = 0.49\linewidth]{sphere_old.pdf}
%\includegraphics[width = 0.49\linewidth]{wide_old.pdf}
%\includegraphics[width = 0.49\linewidth]{sphere_new.pdf}
%\includegraphics[width = 0.49\linewidth]{wide_new.pdf}
%
%\caption{Navigation Function Flows. The left plots show a sphere world that can be navigated with a navigation function $\varphi_k(x)$ of the form specified in \eqref{equ:nav_fun}. The obstacle adds a critical point that can be made to be a saddle by choosing sufficiently large $k$ (left top). The right plots show a space that can't be navigated with a function of this form. No matter how large $k$ we always have an attractive minima (right top). This happens because the obstacle is too wide relative to the level sets of the potential function. The bottom two figures show the proposed corrected flow which eliminates the local minima. A level set of the objective function is shown in red.}
%\label{fig:prob_form}
%\end{figure}

%!TEX root = autosam.tex

\section{Curvature Corrected Navigation Fields} \label{sec_correction}
%As previously stated, an agent that follows the gradient of a navigation function converges to the goal while avoiding the obstacles.
To gain some intuition about the gradient-following dynamics of Rimon-Koditschek potentials we write them explicitly as
\begin{equation} \label{equ:nav_fun_grad}
\begin{split}
\dot x &= -\nabla \varphi_k \\
& =-\left( f_0^k + \beta\beta_0\right)^{-1 - \frac{1}{k}}\left( \beta\beta_0\nabla \fz - \frac{\fz\nabla (\beta \beta_0)}{k}\right).  
\end{split}
\end{equation}
\black{Note that for convenience in notation, we omit the dependence on $x$.} In practice, the dynamics are typically normalized since the norm of the gradient is generally small \citep{whitcomb1991automatic}. Therefore, it is reasonable to omit the scaling $(f_0^k + \beta\beta_0)^{-1 - 1/k}$.  We also omit $\beta_0$ for simplicity, a minor modification which we explain in Section \ref{sec:omitouter}. The \blue{resulting} dynamics is
\begin{align} \label{old_dynamics}
   \dot x  = g_\textrm{nav}(x) :=  -\beta \nabla \fz 
                    + \frac{\fz}{k}\nabla\beta.
\end{align}
The first term, $-\beta \nabla \fz$, in this dynamical system is \blue{a potential field attracting the agent to the goal}, and the second term, $(\fz/k) \nabla\beta$, is a repulsive field pushing the agent away from the obstacles. When the agent is close to the obstacle $\Oc_i$, the product function $\beta$ takes a value close to zero thereby eliminating the first summand in \eqref{old_dynamics} and prompting the agent's velocity to be almost collinear with the vector $\nabla\beta(x)$. In turn, this makes the time derivative $\dot{\beta}(x)$ positive thus preventing $\beta(x)$ from becoming negative. This guarantees that the agent remains in free space. When the agent is away from the obstacles, the term that dominates is the negative gradient of $f_0(x)$ which pushes the agent towards the goal $x^*$. The parameter $k$ balances the relative strengths of these two potentials.

At points where the attractive and repulsive potentials cancel we find critical points. \black{By choosing large enough $k$}, these points can be made into saddles when \eqref{equ:condition} holds. An important observation here is that the condition is always satisfied when the potential and the obstacles are spherical because in that case the left hand side is $(\lambda_\textrm{max}/\lambda_\textrm{min})\times(\mu^i_\textrm{max}/\mu^i_\textrm{min})=1$. This motivates an approach in which we implement a change of coordinates to render the geometry spherical. The challenge is that the change of coordinates that would render one obstacle spherical is not the same change of coordinates that would render the potential, or any of the other obstacles, spherical. Still, this idea motivates the curvature-corrected dynamics that we present in this section. We emphasize that this exposition is only meant to \blue{motivate introducing the navigation dynamics developed in this paper}. At no point during implementation does the agent need to estimate the Hessians of the obstacles. 

%\black{We} first \black{consider} the case of a single obstacle \black{$\beta_1$} and postpone the general case to Section \ref{before_proof}. In this case the function in \eqref{beta_prod} reduces to $\beta(x) = \beta_1(x)$. 
\subsection{Proposed Navigation Dynamics}
Consider the obstacle gradient term $\nabla \beta$ from \eqref{old_dynamics}. Use the product rule of derivation to write
\begin{equation} \label{equ:nabbeta}
\nabla \beta = \sum_{i = 1}^m \bar \beta_i \nabla \beta_i,~~ \bar \beta_i \triangleq \prod_{j \neq i} \beta_j.
\end{equation}

Apply separate correction terms to the gradient of each obstacle function ($\nabla^2 \beta_i^{-1}\beta_i$), which simplifies to $x-x_i$ as well as the potential function ($\nabla^2 \fz^{-1}\fz$), which simplifies to $x-x^*$. Our proposed dynamics therefore becomes
\begin{equation}
\dot x = g_\textrm{new}(x) := -\beta \cdot(x-x^*) + \frac{\fz(x)}{k}\sum_{i = 1}^m \bar \beta_i \cdot (x-x^i).
\label{our_dynamics}
\end{equation}

Similar to how Newton's Method uses second order information to render the level sets of the objective function into spherical sets so to obtain a faster rate of convergence\citep[Ch. 9.5]{boyd2004convex}, pre-multiplying each gradient in the original flow \eqref{old_dynamics} by the Hessian inverse of the corresponding function \emph{corrects} the dynamics so that the world appears spherical to the agent. 
%
%Observe that since \black{both} the objective \black{and the obstacle functions} \black{are} quadratic, the Hessian inverse gradient product in \eqref{one_dynamics} \black{simplifies to}
% \begin{equation}\label{one_dynamics_simplified}
%   \dot x = g_1(x):= -\beta_\black{1}\cdot(x - x^*) + \frac{\fz}{k} (x - x_\black{1}).
% \end{equation}
%These dynamics are almost equivalent to \eqref{old_dynamics} if the dynamics from the navigation function gradient is particularized to an environment with a spherical objective function and a spherical obstacle. Since such a spherical environment satisfies the geometric condition in \eqref{equ:condition}, it is reasonable to expect that a point agent following \eqref{one_dynamics}, which reduces to \eqref{one_dynamics_simplified}, converges to the goal $x^*$ in all environments.
%It is important to note that the correction is not equivalent to treating everything like spheres. This is because the Hessians of the functions $\beta$ and $\fz$ are not necessarily the identity $I_n$. 

An interesting side effect of the simplified dynamics expression in \eqref{our_dynamics} is that implementation is simpler than it appears. The first term pushes in the direction of the goal and the second term pushes away from the center of the obstacle. Thus, the algorithm can be made to work if we just estimate these two quantities. Curvature estimates are not needed for implementation. %The proposed dynamics only consider the case where there is one obstacle. In the following section, we will generalize the dynamics so that the agent can navigate around an arbitrary number of obstacles. Before doing so, we present a pertinent remark on the exclusion of the outer obstacle function $\beta_0$ in our dynamics. 
%
%\black{I would have the Remark regarding $\beta_0$ here, since otherwise the reader needs to wait to much to get to it. Instead of the current version that is based on the proof you can say that for large enough $k$ the dynamics proposed are colinear with $-\nabla f_0(x)$ in the external boundary of the the free space. And say that since the workspace is convex $-\nabla f_0(x)$ points inside the obstacle. If we can formalize it better it's good, but at least to give an intuition along those lines. This also relates to the fact that we observe in the numerical that for small $k$ it breaks a lo}t 
%

%%%%%%%%%%%%%%%%%%%%%%%%%%%%%%%%%%%%%%%%%%%%%%%%%%%%%%%%%%%%%%%%%%%%%%%%%%%%%%%%
%%% S U B -  S   E   C   T   I   O   N   %%%%%%%%%%%%%%%%%%%%%%%%%%%%%%%%%%%%%%%%%%%%%%
%%%%%%%%%%%%%%%%%%%%%%%%%%%%%%%%%%%%%%%%%%%%%%%%%%%%%%%%%%%%%%%%%%%%%%%%%%%%%%%%
%
%\subsection{Extension to Multiple Obstacles} \label{before_proof}
%In this section we extend our proposed dynamics to the case where the workspace is populated with a finite number of non-intersecting obstacles. 
%
%\black{Similar to the one obstacle case, we apply separate correction terms to the gradient of each obstacle function ($\nabla^2 \beta_i\beta_i$), which simplifies to $x-x_i$. Our proposed dynamics therefore becomes}
%

The advantages of this approach are threefold: (i) the estimate of the Hessian does not need to be computed, (ii) the dynamics are simpler and easier to implement, and (iii) the convergence proof is complete for almost all initial conditions $x_0$ (with measure one)\blue{ with mild conditions on the location of the target.} We now present our main result, which guarantees convergence to the target in environments with ellipsoidal obstacles.

\subsection{Convergence Guarantees}
Before we present our main result, we require the following definitions. 
Because the obstacles are ellipsoidal, and the target is a point, for each obstacle, we define the generalized Voronoi cell \citep{arslan2016sensor}
based on maximum margin separating hyperplanes to be
\blue{
\begin{equation} \label{equ:vcell}
\mathcal{C}_i := \left\{ x\in \F\bigg|  \forall_{j\neq i}\forall_{p\in \Oc_i}\forall_{q\in \Oc_j} \|x-p\| \leq \|x - q\|\right\},
\end{equation}
for all $i, j = 1, \dots, m$ where we define $\Oc_0:= x^*$. 
%We similarly define the Voronoi cell for the target which we denote by $C_0$ for notational convenience
%\begin{equation*}
%\mathcal{C}_0 := \left\{ x\in \F\bigg|  \forall_{i}\forall_{p\in \Oc_i} \|x-x^*\| \leq \|x - p\|\right\}.
%\end{equation*}
}
For each Voronoi cell $\mathcal{C}_i$, we can define its boundary to an adjacent cell $\mathcal{C}_j$ by
\blue{
\begin{equation*}
\partial \mathcal{C}_i^j := \left\{ x\in  \mathcal{C}_i\bigg|  \forall_{p\in \Oc_i}\forall_{q\in \Oc_j} \|x-p\| = \|x - q\|\right\}.
\end{equation*}
}
%The following lemma induces a relative ordering on the obstacles, by showing that the agent can only move in one direction on the boundary $\partial \mathcal{C}_i^j$. 
\blue{
We also assume that the target is bounded away from the union of the affine extension of the Voronoi borders. Formally, we assume the following. %Define the normal vector $\n_{i,j}$ to be such that
%
%\begin{equation} \label{equ:normal_on_boarder}
%    \n_{i,j}^\top x\vert_{x\in \partial C_i^j} = 0, \textrm{and~} \n_{i,j}^\top (x-x^*)\vert_{x\in \partial C_i^j} > 0.
%\end{equation}

\begin{assumption} \label{as:target_location} {\bf Target lies away from hyperplanes } For each $\partial \Cc_i^j$, there exists a $\delta_{i,j}> 0$ such that the inner product between $(x-x*)$ and the normal unit vector perpendicular to $\partial \Cc_i^j$, or $\n_{i,j}$, is strictly positive. In particular, for any $a, b,x \in \partial \ccalC_{i}^j$, the normal vector $\n_{i,j}$ satisfies 
\begin{equation} \label{equ:normal_on_boarder}
    \n_{i,j}^\top (a - b) = 0, \textrm{and~} \n_{i,j}^\top (x-x^*) > 0,
\end{equation}
and 
\begin{equation}\label{equ:target_loc}\n_{i,j}^\top (x-x^*) \geq \delta_{i,j}.
\end{equation}
\end{assumption}
}
Assumption \ref{as:target_location} simplifies the analysis significantly, albeit at the cost of placing minor conditions on the location of the target. We discuss the procedure to lift these conditions in Remark \ref{rem:adjacent}. We are now equipped to present our main result.
\begin{theorem} \label{thm:main_result}
Let $\fz(x)$ be a quadratic potential as in \eqref{eqn_goal_def} and let $\beta_i(x)$ be an ellipsoid as in (\ref{equ:beta1_def}) for all $i = 1, \dots, m$. Further let $x$ be the solution of the dynamical system (\ref{our_dynamics}) with initial condition $x_0$\blue{, and let Assumption \ref{as:target_location} be in effect.} Then,  there exists a $K$ such that when $k > K$, $x(t) \in \F$ for all $t \geq 0$ and $\lim_{t\to \infty} x(t) = x^*$ \blue{for almost all initial conditions $x_0$}. 
\end{theorem}
The complete proof is presented in section \ref{proof_of_theorem}, however we present a sketch of proof here. \blue{
\begin{itemize}
\item First, we show that the free space $\F$ is invariant (Section \ref{sec:omitouter}). 
%{\color{red}{SP: Actually you first show that the workspace is invariant. }}
%\item Then, we show that the dynamics \eqref{our_dynamics} induce a directed acyclic graph which determines the order in which the Voronoi cells defined in \eqref{equ:vcell} are traversed (Section \ref{}). 
\item {Then, we show that the dynamics \eqref{our_dynamics} induce a directed acyclic graph which determines the order in which the Voronoi cells defined in \eqref{equ:vcell} are traversed. The final cell visited is a cell  containing the target. (Section \ref{sec:isolating} and Section \ref{sec:ordering}). %SP: I don't think the invariance here makes things clear in terms of what this is result establishes. So I think it's better to remove it.  }
}
%\item First, show that there exist nested invariant sets which induce a directed acyclic graph which necessarily determines the order in which the agent will traverse the cells. The final cell the agent will visit is a cell containing the target. (Lemma \ref{lem:invariant})
    \item Furthermore, we show that $x^*$ is stable in $\ccalC_0$, the Voronoi cell containing the target. (Section \ref{sec:global_V}) %the minimizer of a Lyapunov-type function whose derivative is strictly negative for all $\F$ except for a close neighborhood to the obstacles. 
    %\item Finally, we show that in each cell, the agent will almost always converge to the set where the Lyapunov-type function of Lemma \ref{lem:global_lyap} is valid, which can be interpreted as the agent navigating around the obstacle. (Lemma \ref{lem:local_lyap} and Lemma \ref{lem:unstable})
    \item Finally, we show that the agent will almost always exit each Voronoi cell $\ccalC_i$, $i = 1, \dots, m$. (Section \ref{sec:local_escape})
\end{itemize}
}
The order of the bullets is selected intentionally. To prove bullet three, we use a Lyapunov function. Together with the invariance of the final $\ccalC_0$, it follows that $x^*$ is asymptotically stable. In order to show the final bullet of the proof, we invoke the same Lyapunov function of bullet three in specific regions where we know how the obstacle will traverse the border of those sets. In particular, we use the Lyapunov function to show that the agent will leave that specified set. 
%Note that the Lyapunov function used to prove bullet three is invoked in bullet four to show that the agent passes certain sets without getting 
%{\color{red}{SP: My second (third after you include the invariance of the workspace) item here would be that we show that you always leave a cell (except for the one that contains the target)  }}

%{\color{red}{SP: My last item here would be that in the target cell there is a Lyapunov function.   }}

%\red{SP: After going through the proof again I see why you want to prove first that there exists a global Lyapunov function. You also use that Lemma when discussing the cone. In a sense I think that what doesn't make me happy is that it's not really a global Lyapunov function, but something that you use in a few places to show that you move from a region to the next one. Basically to show that you cannot be stuck forever inside the cone. And to show that once you are in the cell of the target you converge to it. }

%\red{SP:I'm also thinking that maybe we can move the profs that are purely algebraic to an appendix. For instance the proofs of the $\dot{V}$, $\dot{V}_i$ being negative and the proof of the unstable equilibrium. } 

\blue{
\begin{remark} \label{rem:adjacent}
Notice that $\partial \Cc_{i,j}$ are contained in an affine hyperplane. For the purpose of this remark, let $\partial \Cc_{i,j}$ be the affine extension of the border between two adjacent cells. The analysis holds for all configurations except for the case where the target lies on the set with zero measure of the union of the affine hyperplanes $ \cup_{i,j} \partial \Cc_{i,j}$, albeit with arbitrarily high $K$. For that reason, we introduce Assumption \ref{as:target_location} so that we can find a bounded $K$. Theoretically, one should be able to obtain a finite value of $k$ without Assumption \ref{as:target_location} by treating the union of adjacent cells where the border is aligned with $(x-x^*)$ as one cell. In practice, one can consider perturbing the target location within $\ccalC_0$ until Assumption \ref{as:target_location} holds and returning to the original target location when the agent is sufficiently close. As such, we choose to omit this formal discussion because these are corner cases, and they have no effect on the performance in practice. 
\end{remark}
}

\section{Proof of Theorem \ref{thm:main_result}}\label{proof_of_theorem}
In this section, we present the proof of Theorem \ref{thm:main_result}.
Let $N_r(A)\triangleq\bigcup_{p\in A}B_r(p)$, where $B_r(p)\triangleq\{q\colon\Vert q-p\Vert<r\}$ be the \blue{open} $r$-neighborhood of a set $A$.
Further, consider a subset of the free space outside the $\varepsilon$-neighborhood of \emph{any} obstacle, namely,
\begin{equation}
    \mathcal{F}_{>\varepsilon} := \mathcal{F} \backslash  \bigcup_i N_\varepsilon (\mathcal{O}_i).
\end{equation}
\subsection{Invariance of the workspace} \label{sec:omitouter}
\blue{The following Lemma explains why it is possible to omit $\beta_0$ from the dynamics in \eqref{our_dynamics}. %First, we define the following constant 
%\begin{equation} \label{equ:betai_lower_bound}
%    \lambda_0 = \min_{i,x\in \partial \X} \beta_i
%\end{equation}
}
We define the following constants
\begin{equation} \label{equ:betai_lower_adjacent}
    \lambda = \min_{i,x\in \cup_{i,j} \partial \Cc_{i,j}} \beta_i, ~~~ \lambda_0 = \min_{i,x\in \partial \X} \beta_i.
\end{equation}
\begin{lemma}\label{beta_0lemma}
Let $K_{\beta_0}=mP_0\Lambda_0^{m-1}B/\lambda_0^m$, where the $m$ is the number of obstacles and the constants $P_0, \Lambda_0,\lambda_0,$ and $B$ come from \eqref{equ:betai_fz_bound}, \eqref{equ:betai_lower_adjacent}, and \eqref{equ:distance_bound}. For $k>K_{\beta_0}$, $\F$ is invariant under \eqref{our_dynamics}.
\end{lemma}
\begin{proof}
First consider the boundary of an obstacle $\partial\mathcal{O}_i$, for $i = 1, \dots, m$. \blue{Evaluate $\nabla \beta_i^\top \dot x$ with $\dot x$ subject to $\beta_i = 0$. We obtain
\begin{equation}
    \nabla \beta_i^\top \dot x\Big\vert_{\beta_i = 0} =\frac{\fz}{k}\nabla \beta_i^\top(x-x^i),
\end{equation}
} 
which is strictly positive by $\mu_{\min}^i > 0$.

Next, consider the outer obstacle. Without loss of generality, let $\nabla^2\beta_0 = I_n$.
Consider the normal vector $\n = (x-x^*)/\|x-x^*\|$. Evaluate $\n^\top \dot x$ to obtain 
\blue{
\begin{equation} 
\n^\top \dot x = -\beta + \frac{\fz}{k}\sum_{i = 1}^m \bar \beta_i \cdot \n^\top (x-x_i).
\end{equation}
%Since the the $\beta_i$ are positive and finite on the outer boundary $\partial \X$, we bound their value by 
%\begin{equation}
%  \lambda_0 \leq  \beta_i \leq \Lambda_0,
%\end{equation}
%for all $i = 1, \dots m$. We similarly bound the objective function
%\begin{equation}
%  p_0 \leq  \fz \leq P_0.
%\end{equation}
Recalling the bounds \eqref{equ:distance_bound} and \eqref{equ:betai_fz_bound}, $\n^\top \dot x$ is strictly negative when $k > K_{\beta_0}$ defined by
\begin{equation} \label{Kbeta0}
    K_{\beta_0} :=\frac{mP_0\Lambda_0^{m-1}B}{\lambda_0^m}.
\end{equation}
We invoke  Nagumo's Theorem \citep{blanchini1999set}} to complete the proof.
\end{proof}
%
%\blue{The remainder of the proof of Theorem \ref{thm:main_result} proceeds as following. First, we define Voronoi cells around each obstacle and the target. We then show that the agent can only traverse from one cell to the next in a specific relative ordering, which induces a order on the obstacles. Finally, we show that $x^*$ is asymptotically stable in each Voronoi cell, which guarantees that the agent will navigate around a finite number of obstacles until it reaches the target. }
%
\subsection{Isolating obstacles analysis with Voronoi cells} \label{sec:isolating}
\begin{figure}[!t]
\centering
\includegraphics[trim = {2.5cm, 7.5cm, 2.5cm, 7.5cm}, clip,width=\linewidth]{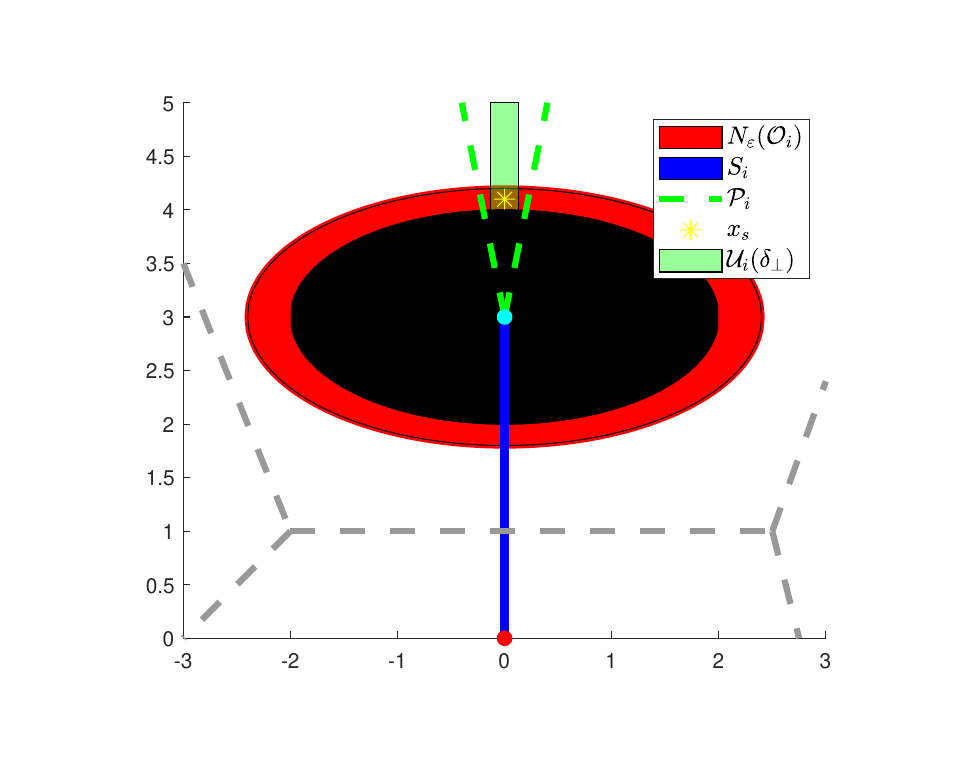} 
\caption{Visualization of the different regions of the spaces described in the proof. The red region around the obstacle in black is the $\varepsilon$ neighborhood of the obstacle. The green dashed lines and the green cylinder define the pyramid and cylinder discussed in Section \ref{sec:local_escape}. The blue line describes $S_i$, which attracts the points in the Voronoi cell whose borders shown by the dashed gray lines. }
\label{zone_vis}
\end{figure}
%In this section, we will show that the agent will visit the Voronoi cells in a specific order-that is, once the agent leaves a Voronoi cell, it will never return to that cell. In particular, the agent will end in a Voronoi cell which contains the target. %This ordering serves to show that once the agent leaves a Voronoi cell, it will never return to that cell. 

%\red{The ordering depends on the position of the goal right? I don't see that explicit. If I look at $n_{j,i}$ the product should be positive right?  }

Recall from Assumption \ref{as:target_location} that $\n_{i,j}$ is a unit vector perpendicular to the border $\partial \ccalC_{i,j}$ such that $\n_{i,j}^\top(x-x^*) > \delta_{i,j}$ for some positive $\delta_{i,j} > 0$.
\blue{
Our next lemma establishes a relative ordering on adjacent obstacles.
}
\begin{lemma} \label{lem:voronoi}
Let Assumption \ref{as:target_location} hold. Further, let $K_{i,j} = m P_0\Lambda_0^{m-1}B\delta_{i,j}^{-1}\lambda^{-m}$, where $m$ is the number of obstacles and the constants $P_0$, $\Lambda_0$, $\lambda$, $B$, and $\delta_{i,j}$ come from \eqref{equ:betai_fz_bound}, \eqref{equ:betai_lower_adjacent}, \eqref{equ:distance_bound}, and Assumption \ref{as:target_location}. Then, for all $k > K_{i,j}$, $\n_{i,j}^\top \dot x < 0$.
\end{lemma}
\begin{proof}
Evaluate $n^\top \dot x$
\begin{equation}
\n^\top \dot x = -\beta \n^\top(x-x^*) + \frac{f_0}{k}\sum_{j = 1}^m \bar \beta_j\n^\top(x-x_i).
\end{equation}
By the definition of the Voronoi cell, $\partial C_i^j$ does not intersect with any obstacle. Invoke \eqref{equ:distance_bound}, \eqref{equ:betai_fz_bound}, and \eqref{equ:betai_lower_adjacent}
to obtain the bound
\begin{equation}
\beta \n^\top (x-x^*) \geq \delta_{i,j} \lambda^m.
\end{equation}
We also bound the expression
\begin{equation}
\fz\sum_{j = 1}^m \bar \beta_j \n^\top (x-x_j) \leq mP_0\Lambda_0^{m-1}B.
\end{equation}
Using these bounds, we define
\begin{equation} \label{equ:K_from_voronoi}
K_{i,j} :=  \frac{ mP_0\Lambda_0^{m-1}B}{\delta_{i,j} \lambda^m}.
\end{equation}
It holds that $\n^\top \dot x$ is negative for all $k>K_{i,j}$
%By selecting $k > K_{i,j}$, we guarantee that $n^\top \dot x < 0$.
\end{proof}
\blue{In fact, for any hyperplane section that is $\varepsilon$-bounded away from the targets, we can establish a similar traversal property. We formalize this with the following corollary. 
\begin{corollary}\label{cor:hyperplane}
Let $\ccalH$ be a hyperplane section that is contained in $\F_{>\varepsilon}$. Let $n$ be a normal vector to the hyperplane section such that $\n^\top(x-x^*) > \delta$ for some $\delta > 0$ and for any $a, b \in \ccalH$, $n^\top (a-b) = 0$. Then for any $k > mP_0\Lambda_0^{m-1}B\delta^{-1}\lambda^{m}$, it holds that $\n^\top \dot x < 0$ for the flow given in \eqref{our_dynamics}.
\end{corollary}
\begin{proof}
The proof follows identically from Lemma \ref{lem:voronoi}.
\end{proof}
}

Lemma \ref{lem:voronoi} induces a relative ordering on adjacent obstacles. Indeed, this ordering is such that once an agent leaves the Vornonoi cell of an obstacle $\Oc_i$, it will never return to that cell. Eventually, the agent will be in $\Cc_0$ where it will converge to the target (see Section \ref{sec:global_V}). We will first formally define the ordering, then we will prove the claim. 
\blue{
\subsection{Ordering on the obstacles} \label{sec:ordering}
We will construct the ordering $\rho(0), \dots, \rho(m)$ on the Voronoi cells $\mathcal{C}_i$, for $i = 0, \dots, m$. First assign 0 to $\rho(0)$, that is the Voronoi cell containing the target.

Next, define the Voronoi cells for the remaining obstacles $\mathcal{C}_i(\ell)$, where $\ell$ denotes the number of obstacles removed. 
For brevity in notation, we let $\rho_\ell$ denote $\left\{\rho(i)\right\}_{i = 0}^\ell$. We define
\begin{equation} \label{equ:vcell_minus_ell}
\mathcal{C}_i(\ell) := \left\{ x\in \F\bigg|  \forall_{j\notin \rho_\ell \atop j \neq i}\forall_{p\in \Oc_i}\forall_{q\in \Oc_j} \|x-p\| \leq \|x - q\|\right\}.
\end{equation}
Where $\mathcal{C}_i(\ell)$ is defined only for $i = 1, \dots, m$ such that $i\notin \rho_\ell$ and $\ell = 0, \dots, m-2$. By Assumption \ref{as:target_location}, there exists an $i\notin \rho_\ell$ such that $x^* \in \mathcal{C}_i(\ell)$. Assign this $i$ to $\rho(\ell+1)$. This process is repeated for all $\ell$ until there are are just two cells. Assign $i\notin \rho_{m-1}$ to $\rho(m)$ to complete the ordering. This process is shown visusally in Figure \ref{voronoi_ordering}.
}
\begin{figure*}[!t]
\centering
\begin{tabular}{ccc}
\includegraphics[width=.3\linewidth]{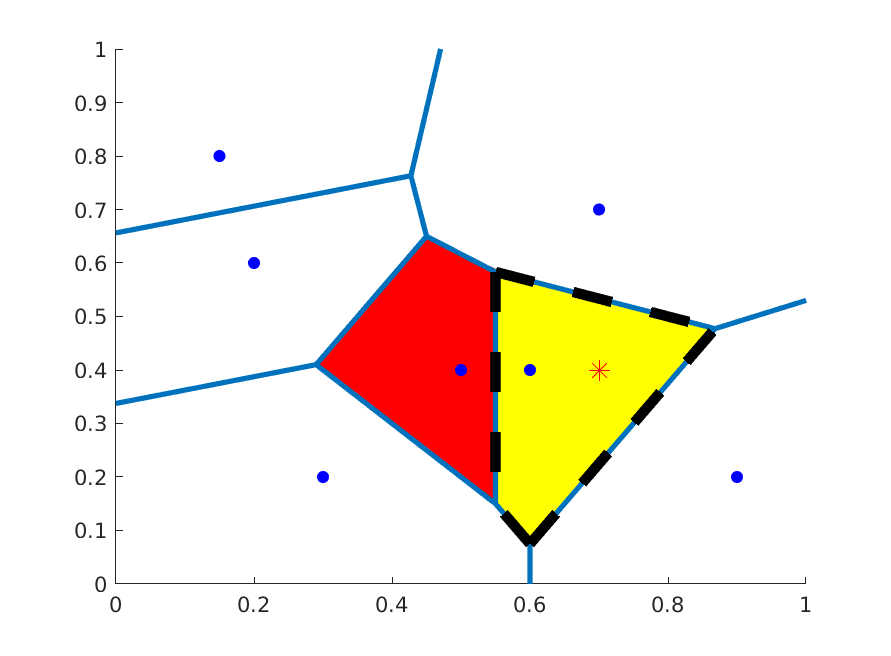} &
\includegraphics[width=.3\linewidth]{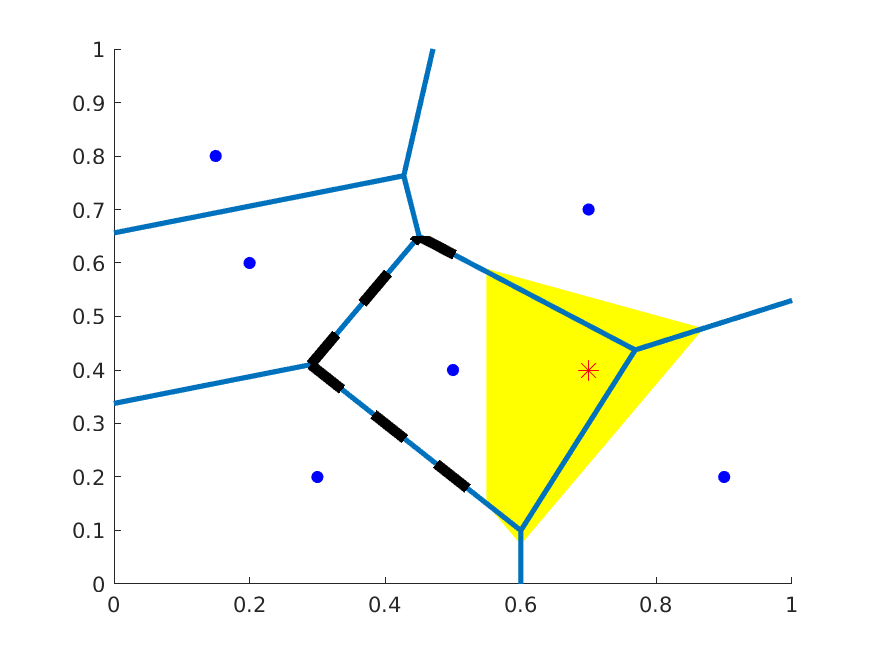} &
\includegraphics[width=.3\linewidth]{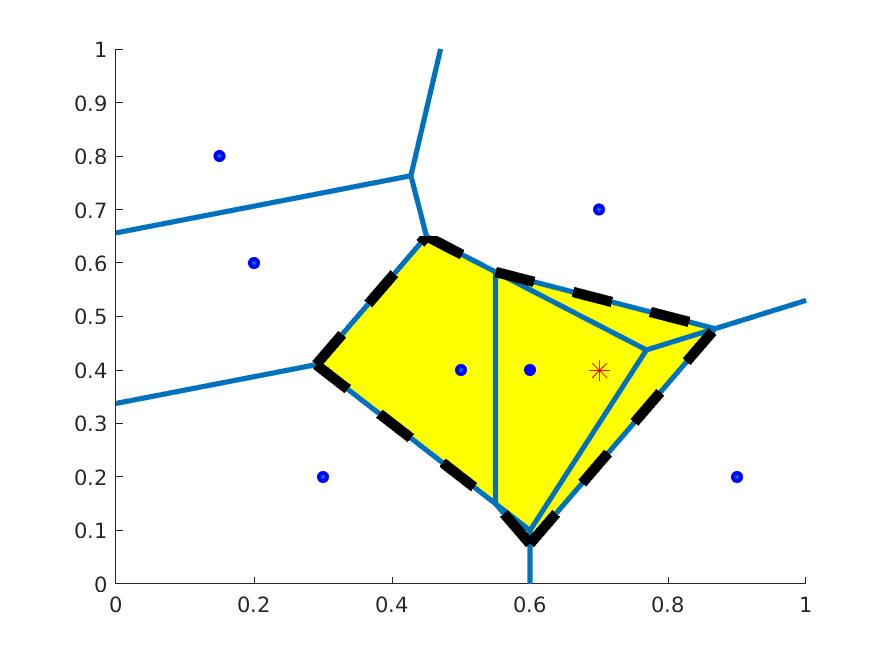}\\
\small (a) & \small (b) & \small (c)
\end{tabular}
\caption{Let the red $*$ represent the location of the target. (a) The cell containing the target is invariant by Lemma \ref{lem:voronoi}, as shown with the dashed border. The red cell is the next obstacle we will add to the invariant set. (b) Consider removing the yellow cell from Fig 2a, and reconstruct the Voronoi diagram. The obstacle lies in the new Voronoi cell corresponding to the red Voronoi cell. The dashed borders are the same in the original Voronoi diagram, and by Lemma \ref{lem:voronoi}, the agent can only enter the red cell from the white neighboring cells. (c) As shown in Lemma \ref{lem:invariant}, the union of the two cells is invariant.}
\label{voronoi_ordering}
\end{figure*}
\blue{
The following lemma establishes that the union $\cup_{i = 0}^n \mathcal{C}_{\rho(i)}$ is invariant for all $n = 0, \dots, m$. 
\begin{lemma} \label{lem:invariant} Let $K_\rho = \max_{i, j} K_{i,j}$, where $K_{i,j}$ is defined as in \eqref{equ:K_from_voronoi}. Let $\rho$ be the ordering defined in Section \ref{sec:ordering}. Then, for all $k>K_\rho$, it holds that $\cup_{i = 0}^n \Cc_{\rho(i)}$ is invariant under the flow \eqref{our_dynamics} for all $n = 0, \dots, m$. 
\end{lemma}
}
\blue{
\begin{proof}
We prove the lemma by induction. That $\Cc_0$ is invariant comes directly from applying Lemma \ref{lem:voronoi} on $\partial \Cc_0$ via Nagumo's Theorem \citep{blanchini1999set}. 

Next, we will establish that 
\begin{equation}\Cc_{\rho(\ell+1)}(\ell) \cup \left(\cup_{i = 0}^\ell \Cc_{\rho(i)}\right) = \cup_{i = 0}^{\ell+1} \Cc_{\rho(i)}. \end{equation}
It suffices to show equivalence on the set difference
\begin{equation} \label{equ:equal_sets}
\begin{split}
\Cc_{\rho(\ell+1)}(\ell) \backslash \left(\cup_{i = 0}^\ell \Cc_{\rho(i)}\right) &= \Cc_{\rho(\ell + 1 )}\backslash \left(\cup_{i = 0}^\ell \Cc_{\rho(i)}\right)\\
&= \Cc_{\rho(\ell +1)} \backslash \cup_{j \in \rho_\ell} \partial \Cc_{\rho(\ell + 1)}^j,
\end{split}
\end{equation}
where the second equality holds because the intersection of two adjacent Voronoi cells only contains the border. %, so we can equivalently write the right hand side of \eqref{equ:equal_sets} by

%\begin{equation} \label{equ:equal_sets_rhs} \Cc_{\rho(\ell + 1 )}\backslash \left(\cup_{i = 0}^\ell \Cc_{\rho(i)}\right) \equiv \Cc_{\rho(\ell +1)} \backslash \cup_{j \in \rho_\ell} \partial \Cc_{\rho(\ell + 1)}^j .
%\end{equation}
%
%\red{SP: This step needs a bit more detail. What are the cells that are adjacent?} 

Consider the left hand side of \eqref{equ:equal_sets}. By definition, it must be that both
\begin{equation} \label{equ:set_equiv}
\forall_{j \notin \rho_\ell \atop j \neq \rho(\ell + 1)} \forall_{p \in \Oc_{\rho(\ell + 1)}} \forall_{q \in \Oc_j} \|x-p\| \leq \|x-q\|, \end{equation}
by \eqref{equ:vcell_minus_ell}, and
%\red{SP: Shouldn't $i$ be $\rho(\ell+1)$ in the first condition? }
\begin{equation}
    \forall_{j \in \rho_\ell} \forall_{p \in \Oc_{\rho(\ell + 1)}} \forall_{q \in \Oc_j} \|x-p\| < \|x-q\|,
\end{equation}
by \eqref{equ:vcell}.
%
%\red{SP: These two may need a bit more explanation too}
This is precisely the definition for the right hand side of \eqref{equ:equal_sets}. By the induction hypothesis, and by the fact that $x^* \in \textrm{int}(\Cc_\rho(\ell + 1)(\ell))$ by the construction of $\rho$, it follows from Lemma \ref{lem:voronoi} and Nagumo's theorem that $\cup_{i =0}^{\ell + 1} \Cc_{\rho(i)}$ is invariant. 
\end{proof}
Given Lemma \ref{lem:invariant}, we can define a Directed Acyclic Graph (DAG) on the Voronoi cells which necessarily defines the order in which the target will visit obstacles. 
Let $\mathcal{G}$ be the graph whose nodes consist of the Voronoi cells $C_i$. We build the DAG by adding an edge to the graph $\mathcal{G}$ between $i,j$ whenever $\rho(i) < \rho(j)$.
}

\subsection{Stability of the target} \label{sec:global_V}
Because the obstacles are compact and disjoint, there exists $\varepsilon_0$ such that all $N_\varepsilon(\Oc_i)$ are contained in their corresponding Voronoi cell $\ccalC_i$ and $x^* \in \textbf{int}(\F_{>\varepsilon})$ for all $\varepsilon < \varepsilon_0$. Formally, let
\begin{equation} \label{equ:var_ep_0}
\begin{split}
    \varepsilon_0 = \arg\max_{\varepsilon} & ~~~~\varepsilon \\  
    \textrm{subject to~}& N_\varepsilon(\Oc_i) \subset \ccalC_i, \forall i = 1, \dots, m\\
    & N_\varepsilon (x^*) \subset \ccalC_0 
\end{split}
\end{equation}
Then, within $N_{\varepsilon}(\mathcal{O}_i)$, consider the points on the same side of the obstacle as the target. Formally, define 
\blue{
\begin{equation} \label{equ:si}
    S_i := \left\{x \in \ccalC_i \given \frac{(x-x^*)^\top (x-x_i)}{\|x-x^*\|\cdot \|x-x_i\|} =   -1 \right\}
\end{equation}
}

On the union of these sets, we define a global Lyapunov function candidate $V:\mathcal{F}_{>\varepsilon}\cup ( \cup_i S_i )\to \mathbb{R}$
by
\begin{equation}\label{eqn_lyap_function}
V(x) = \frac{1}{2}\|x-x^*\|^2.
\end{equation}
Note that the Lyapunov function candidate can be selected to be $\fz$. Without loss of generality and for simplicity, we consider the from in \eqref{eqn_lyap_function}. By definition, $V$ is always positive, and it is equal to zero only when $x = x^*$. In the following lemma, we show that $\dot V <0$ in $\mathcal{F}_{>\varepsilon}\cup ( \cup_i S_i )$. Before we present the lemma, we define the following constants.

Similar to \eqref{equ:betai_lower_adjacent}, we define the following constants 
\begin{equation} \label{equ:betai_lower_epsilon}
    \lambda_\varepsilon := \min_{i,x\in \mathcal{F}_{> \varepsilon}} \beta_i ~ \textrm{and }~\lambda_0 := \min_{i,x\in N_{\varepsilon_0}(x^*)} \beta_i.
\end{equation}
%Further, define 
%\begin{equation} \label{equ:hi_bound}
%    H_i := \min_{i,x\in S_i \cup N_{\varepsilon}(\ccalO_i)} \|x-x^*\|\cdot \|x-x_i\|.
%\end{equation}

\begin{lemma} \label{lem:global_lyap} Choose $\varepsilon \in (0, \min \left\{\lambda_0  m^{-1}B^{-1}\Lambda_0^{-1}, \varepsilon_0\right\})$, where $B$ and $\Lambda_0$ come from \eqref{equ:distance_bound} and \eqref{equ:betai_fz_bound}, and $\varepsilon_0$ comes from \eqref{equ:var_ep_0}. Let $K_\varepsilon := B\lambda_\varepsilon m$. Then, for all $k > K_\varepsilon$ and $x \in \mathcal{F}_{>\varepsilon} \cup_i S_i$, $\dot V <0$ for $V$ defined in \eqref{eqn_lyap_function}.
\end{lemma}
\begin{proof}
See Appendix \ref{appendix:global_lyap}
\end{proof}

Lemma \ref{lem:global_lyap} shows that $\dot V < 0$ for all $x\in \F_{>\varepsilon}\cup \left(\cup_i S_i\right)$. The red region in Figure \ref{zone_vis} shows the set close to an obstacle where $\dot V < 0$ does not necessarily hold. The blue region represents the set $S_i$.

What remains to be shown is that the agent will navigate around the obstacle toward the goal in each cell.

\subsection{Navigating around the obstacle} \label{sec:local_escape}
We show that the agent will navigate around an obstacle $\mathcal{O}_i$ in two steps. First, we consider the angle between the vectors $x-x_i$ and $x_i-x^*$ is increasing using a local Lyapunov function candidate. Second, we show that there is an unstable equilibrium on the side of the obstacle opposite the target. Before we formalize this with the necessary lemmas, we must introduce the following definitions.

We begin by writing $x-x_i$ as the sum of its parallel and perpendicular components to $x_i-x^*$. In particular, we let 
\begin{equation}\label{equ:decomp}
    x-x_i = a(x_i - x^*) + \rhop,
\end{equation}
where $\rhop^\top (x_i-x^*) = 0$ and $a \in \mathbb{R}$.
%
%By selecting the size of $\|\rhop\|$, for any $\delta_\perp > 0$, we define the following set
\blue{Consider the following set}
\begin{equation}
    \ccalU_i(\delta_\perp) := \left\{x\in \Cc_i \given a>0,  \|\rhop\|_\infty\leq \delta_\perp  \right\},
\end{equation}
where $\delta_\perp$ is strictly positive.
%\red{SP: I'm confused by this definition. How is this a cone? $p$ is the perpendicular to the direction $x_i-x^\star$. Isn't this more like a rectangle? Can you not define the cone directly in terms of the function $V_i$? That would define some angle and now you do have a cone. }
\blue{Because $a>0$, it follows that this set is defined on the side of the obstacle opposite the target. Next, consider a right regular pyramid with a vertex at $\xi:=\gamma x_i + (1-\gamma) x^*$ for some $\gamma$. For the purpose of this proof, we choose $\gamma$ to be equal to $3/4$. Define $\ccalP_i$ to be 
\begin{equation}
\begin{split}
    \ccalP_i := \arg \min_{\ccalP \in \Xi}& ~~A(\ccalP)\\
    \textrm{subject to} &~~ \ccalU_i(\delta_\perp) \subseteq \ccalP,
\end{split}
\end{equation}
where $A(\cdot)$ denotes the volume of the set and $\Xi$ is the set of regular right pyramids whose vertex is $\xi$. By the definition of the infinity norm $\|\cdot\|_\infty$, the right regular pyramid will have $2(n-1)$ faces, where we recall $n$ is the dimension of $\ccalX$. 

Given $\ccalP_i$, we can consider the $\varepsilon$-frustum, which is the intersection $\ccalP_i\cap \F_{>\varepsilon}$. First consider the faces of the $\varepsilon$-frustum. By choice of $\xi$, it follows that there exists a $\delta_i >0$ such that for each normal vector $\n_{i, \ell}$ to the face $\partial \ccalP_{i,\ell}$ satisfies $\n_{i,\ell}^\top(x-x^*) > \delta_i$ for all $\ell = 1, \dots, 2(n-1)$. Further, since we are bounded away from $\cup_i N_\varepsilon(\Oc_i)$, we can apply Corollary \ref{cor:hyperplane} so that when $K_{i, \ell}> m P_0 \Lambda_0^{m-1}B\delta_{i}^{-1}\lambda^{-m}$, it follows that the agent can only traverse the faces of the pyramid by exiting $\ccalP_i$. 

Next we consider the $\varepsilon$-boundary of the obstacle, and show that the agent can only traverse the boundary by entering $N_\varepsilon(\Oc_i)$. We formalize this with the following corollary. 
\begin{corollary} \label{cor:pyramid}
Let $K = P_0B(\varepsilon^{-1} + (m-1)\lambda_\varepsilon )$. Then on the set $\ccalP_i \cap \partial \F_{>\varepsilon}$, when $k> K$, $(x-x^*)^\top \dot x/\|x-x^*\| < 0$
\end{corollary}
\begin{proof}
The proof follows the arguments made in Lemma \ref{beta_0lemma} with the substitution of $\varepsilon$ for $\beta_i$ and the bound \eqref{equ:betai_lower_epsilon} instead of \eqref{equ:betai_lower_adjacent} for $\beta_j$, for $j \neq i$.
\end{proof}
}
%The complement of $\ccalU_i(\delta_\perp)$ inside $\Cc_i$ defines an area like a cone as shown by the green dashed lines in Figure \ref{zone_vis}. We will define this region by $\ccalD_i(\delta_\perp)$, formally
%\begin{equation}
%    \ccalD_i(\delta_\perp) := \left\{x\in \Cc_i \given  \|\rhop\|< \delta_\perp  \right\}.
%\end{equation}
The analysis is broken into two steps. First we consider the region outside the pyramid, namely  $\ccalC_i\backslash \ccalP_i$. We define a local Lyapunov Function Candidate %on the set $\mathcal{U}_{i}(\delta_{\perp}):= \left\{ x\in N_\epsilon(\mathcal{O}_i) \vert \|\rho\| > \delta_{\perp} \right\}$, namely 
$V_i: \mathcal{C}_i\to \mathbb{R}$
\begin{equation}
V_i := 1 + \frac{(x-x_i)^\top (x_i-x^*)}{\|x-x_i\|\cdot \|x_i-x^*\|}
\end{equation}
With the following lemma, we establish that $V_i$ acts as a local Lyapunov function candidate by showing that $\dot V_i$ is strictly negative on $\ccalP_i^c$. \blue{In particular, $V_i$ describes that the agent will asymptotically converge to the points where $x-x_i$ and $x_i - x^*$ are pointing in opposite directions. This is precisely the region $S_i$ defined in \eqref{equ:si} (i.e. the blue line in Figure \ref{zone_vis}), where we know from Lemma \ref{lem:global_lyap} that $x^*$ is asymptotically stable. }
Before we present the lemma, consider the following bounds,
\begin{equation} \label{equ:xi_lower_bound}
    \|x-x_i\|^2\big\vert_{x\in \Cc_i} \geq \theta_i,
\end{equation}
which holds because $x_i \notin \F$, and
\begin{equation} \label{equ:bj_lower_bound}
        \lambda_{i,\perp} = \min_{j \neq i,x\in \Cc_i} \beta_j,
\end{equation}
which holds because we are bounded away from the other obstacles. 
%
%\red{TODO: Define $\delta_\perp$}
\begin{lemma} \label{lem:local_lyap} Given $\delta_\perp > 0$, let $K_i^\perp := 2P_0\lambda_{i,\perp} \delta_\perp^{-2} (1 + 1/\theta_i) (a^2B  +\delta_\perp^2)$, where $\theta_i$ and $\lambda_{i,\perp}$ come from \eqref{equ:xi_lower_bound} and \eqref{equ:bj_lower_bound}, $P_0$ and $B$ come from \eqref{equ:betai_fz_bound} and \eqref{equ:distance_bound}, and $a$ comes from \eqref{equ:decomp}. Then, for all $k > K_i^\perp$, $\dot V_i < 0$ defined on $\ccalC_i\backslash \ccalP_i$ (i.e. outside the pyramid).
\end{lemma}

\begin{proof}
See Appendix \ref{appendix:local_lyap}
\end{proof}
%\red{Here I would say that if we are outside of the cone then we don't go back in the cone and that you exit the Voronoi cell. 
Next we consider $\ccalP_i$. In particular, there are three options for the agent in this region. (i) The agent will remain in $\ccalP_i$ and away form $N_\varepsilon(\Oc_i)$. (ii) The agent will move toward $\ccalP_i \cap N_\varepsilon(\Oc_i)$. (iii) The agent will enter $\ccalC_i \backslash \ccalP_i$. 

(i) is impossible because $\dot V<0$ as a consequence of Lemma \ref{lem:global_lyap}. (iii) means that the region enters $\ccalC_i\backslash \ccalP_i$ and follows the behaviour described in Lemma \ref{lem:local_lyap}. Note also that by Corollory \ref{cor:pyramid}, the agent will not be able to return to the pyramid. % \red{The discussion of (iii) is not very clear. I think we need to make explicit that you cannot go back in the cone. I think that this will be more clear if you take into account my previous comment. }We apply Lemma \ref{lem:voronoi} to show that the the agent can only exit the cone away from the $N_\varepsilon(\Oc_i)$. 
That leaves us with (ii). For small enough $\varepsilon$, the dynamics in the small region of $\ccalP_i \cap N_\varepsilon(\Oc_i)$ can be approximated by the linear system with dynamics 
\begin{equation}
    \dot x = J(g_\textrm{new})\given_{x_s}(x-x_s),
\end{equation}
where $x_s$ is a critical point and $J$ is the Jacobian. 
To complete the proof of Theorem \ref{thm:main_result}, what remains to be shown is that there is an unstable equilibrium inside this region. We formalize this with the following lemma.
%
%
%\red{The unstable equilibrium is not in the cone (well it is), but it's in the intersection between the cone and the epsilon neighborhood. That's why you can make the region arbitrarily small. If it were the whole cone then you cannot and the linear approximation may not hold.   }
%
\begin{lemma} \label{lem:unstable}
Choose $\delta_\perp > 0$. There exists a $K_{\rhop, i}$ such that for all $k >K_{\rhop, i}$, following the flow \eqref{our_dynamics} induces an unstable equilibrium on $\ccalP_i\cap N_\varepsilon(\Oc_i)$ (i.e. inside the cone).
\end{lemma}
%We can bound the remaining $j \neq i$ terms by some constant $C_{3}$. That is, 
%\begin{align*}
%&\fz\sum_{j =1, j \neq i}^m\frac{1}{\beta_j} \bigg( (x_i-x^*)^\top(x-x_j) \\
%&-\frac{(x-x_i)^\top(x_i -x^*)(x-x_i)^\top (x-x_j)}{\|x-x_i\|^2} \bigg)\Bigg) \leq C_{3}
%\end{align*}
%As such, for every $\delta_{\perp}$, there exists a $K_i^\perp$ such that for all $k > K_i^\perp$ we have that $\dot V$ is strictly negative. This completes the proof.

%Intuitively, this means the angle between vector defined by the the agent to the center of the obstacle and the vector defined by he center of the obstacle to the target is increasing. 

%Eventually, the agent will either enter $S_i$, or leave the $\varepsilon$-neighborhood of the obstacle, where the asymptotic stability of $x^*$ has been shown to hold. 

%Consider the case where $\|\rhop\| < \delta_{\perp}$. When $\beta_i > \varepsilon$, we know \blue{from Lemma \ref{lem:global_lyap}} that $\dot V$ is strictly negative, and so the agent is moving toward the target and will either enter $\ccalU_i$ or enter $N_\varepsilon(\Oc_i)$. 
\begin{proof}
See Appendix \ref{appendix:unstable}
\end{proof}
To conclude the proof of Theorem \ref{thm:main_result}, we first select $\varepsilon \in (0, \min \left\{\lambda_0  m^{-1}B^{-1}\Lambda_0^{-1}, \varepsilon_0\right\}) $, from Lemma \ref{lem:global_lyap}. Then select $K = \max_i \left\{K_{\beta_0}, K_\varepsilon, K_{\rho},  K_i^\perp,K_{\rhop, i}\right\}$, from Lemmas 1, 3-6.

\section{Numerical Results} \label{num_results}
\begin{figure*}[!t] 
\centering
\begin{tabular}{cc}
\includegraphics[trim = {1.5cm, 1cm, 1.5cm, 0},width=.45\linewidth]{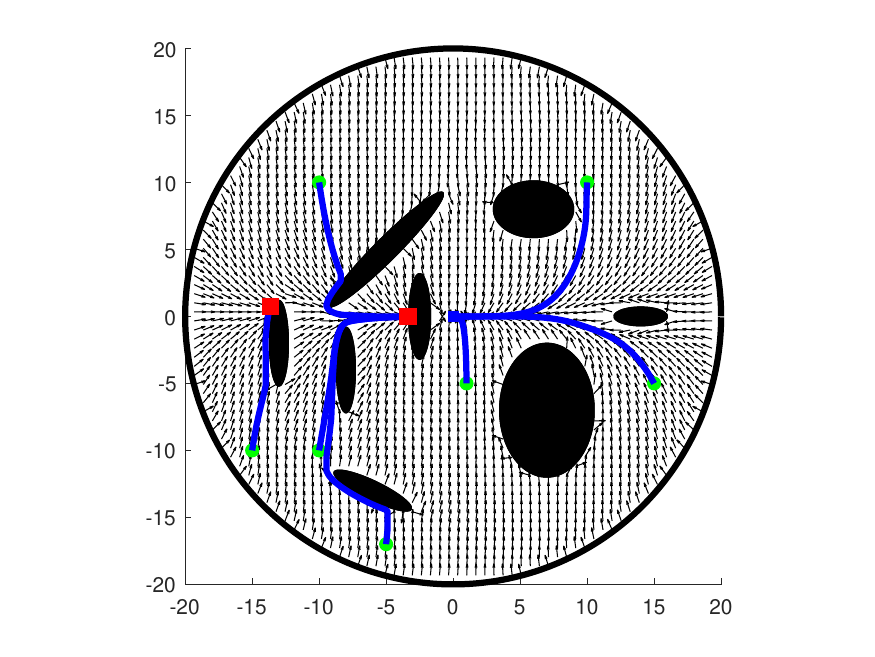} &
\includegraphics[trim = {1.5cm, 1cm, 1.5cm, 0},width=.45\linewidth]{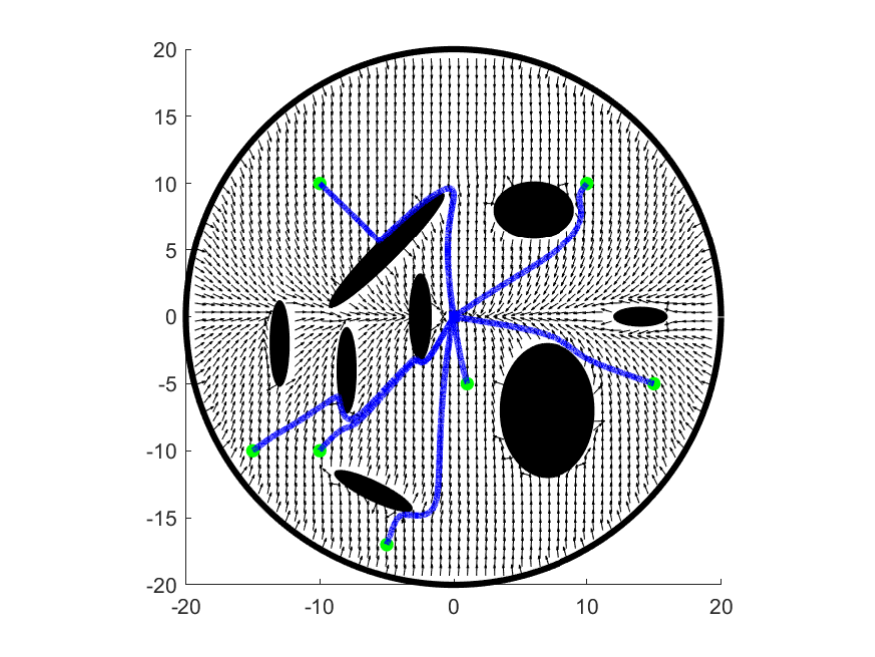} \\
\small (a) & \small (b)
\end{tabular}
\caption{(a)Trajectories generated by following the negative gradient of the RK Potential -- that is $\dot x = g_\textrm{nav}$  -- which is not a navigation function as condition \eqref{equ:condition} is violated. (b)Trajectories generated by following our proposed dynamics, that is $\dot x = g_\textrm{new}$.  Trajectories which converge to a local minimum of $\varphi(x)$ end in a red square. We set $k = 15$ and \blue{$x^*$ is the origin.}}
\label{mult_old_new}
\end{figure*}

In this section, we compare the performance of our proposed dynamics $g_\textrm{new}$ to the performance of navigation function dynamics $g_\textrm{nav}$. We consider a discrete approximation for the flow $\dot x = g(x)$. In practice, the norm of the dynamics is generally very small. This may cause problems numerically when computing the direction as well as taking taking a long time for the agent to reach its target. Hence, what is often used in practice is to normalize the gradient by scaling it by a factor of $(\epsilon + \|g(x)\|)$, where $\epsilon > 0$ \citep{whitcomb1991automatic}. As such the dynamics will be 
\begin{equation} \label{discrete_new}
x_{t+1} = x_t + \eta \frac{g(x_t)}{\|g(x_t)\| + \epsilon} 
\end{equation}
where $\eta$ is a constant step size. We set $\epsilon = 10^{-4}$ and $\eta = 0.01$ in \eqref{discrete_new} for all simulations.

First, we consider a world with eight ellipsoidal obstacles, and we compare the trajectories from several different initial positions. Then, we explore the effect of increasing the number of obstacles of randomly generated ellipsoidal worlds. The obstacles are generated such that condition (\ref{equ:condition}) might fail. Therefore, increasing the number of obstacles results in fewer trajectories following $g_\textrm{nav}$ successfully reaching the target. In contrast, the corrected dynamics performs well even when the number of obstacles is large. 

\subsection{Correcting the Field} \label{sec_traj_plots}
In this section, we show an ellipsoidal world with eight obstacles and several different initialization points. We designed the world such that condition \eqref{equ:condition} does not necessarily hold, thereby eliminating the guarantee that the Rimon Koditschek potential is a navigation function. The maximum to minimum ratio of the eigenvalues of the eight obstacles range between 2 and 50. The radius of the outer obstacle $\beta_0$ is equal to 20. The objective value of the function is chosen to be $\fz(x) = \|x\|^2$. 

Figure \ref{mult_old_new} (a) shows the vector field and some trajectories for the navigation function dynamics $\dot x = g_\textrm{nav}$. Indeed, condition \eqref{equ:condition} is violated. As such, four of the trajectories converge to local minimum appearing behind the obstacles which violate the condition instead of to the target. We selected $k = 15$ because this was the maximum value for $k$ considered in the analysis for worlds which violate the condition \citep{paternain2018navigation}. 

We compare the trajectories of \blue{our proposed dynamics to} the navigation function dynamics where the condition is violated. The Figure \ref{mult_old_new} (b) shows that with the same value of $k = 15$, all of the trajectories converge to the target. The vector field plots show that there is only one stable point, the target located at the origin. This is consistent with Theorem \ref{thm:main_result}.

%The corrected dynamics move toward the target except for when the agent is close to an obstacle. At this point, the agent veers off away from the center of the obstacle as predicted by Lemma \ref{lem:break_cycles}. We emphasize that the trajectories following $g_\textrm{old}$ requires second order information whereas those following $g_\textrm{new}$ does not. Our proposed dynamics achieve the same performance with less information.

\subsection{Obstacles in $\R^2$}

\begin{figure*}[!t]
\centering
\begin{tabular}{cc}
\includegraphics[width=.4\linewidth]{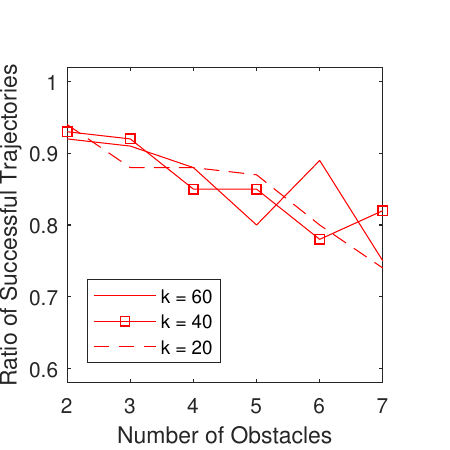} &
\includegraphics[width=.4\linewidth]{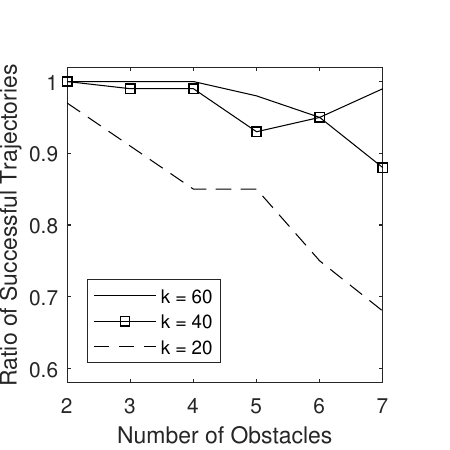} \\
\small (a) & \small (b) 
\end{tabular}
\caption{(a) Uncorrected Dynamics: Regardless of the value $k$, the ratio of successful simulations decreases as the number of obstacles increases. (b) Proposed Dynamics: By increasing $k$, the ratio of successful simulations remains high regardless of the number of obstacles}
\label{m_trend}
\end{figure*}

In this section, we explore the effect of increasing the number of obstacles on the percentage of successful trajectories. We define the external shell to be the a sphere with center $(0,0)$ and radius $r_0$. The center of each ellipsoid is drawn uniformly from $[-r_0/2,r_0/2]^2$. The maximum semiaxis $r_i$ is drawn uniformly from $[r_0/10,r_0/5]$. The positive definite matrices $A_i$ have eigenvalues $1$ and $\mu_i$, where $\mu_i$ is drawn randomly from $[1,r_0/2 ]$. The obstacles are then rotated by $\theta_i$ where $\theta_i$ is drawn randomly from $[-\pi/2,\pi/2]$. The obstacles are redrawn if they overlap. For the objective function, we consider a quadratic cost given by 
\begin{equation}
\fz(x) = (x-x^*)^\top Q (x-x^*).
\end{equation}
where $Q\in \ccalM^{2 \times 2}$ is a diagonal matrix with eigenvalues $\eig(Q) = \{1,\lambda\}$ where $\lambda$ is drawn from $[0,r_0]$. The minimizer of the objective function $x^*$ is drawn uniformly from $[-r_0/2, r_0/2]^2$. The minimizer $x^*$ is redrawn if it is not in the free space. Finally, the initial position is drawn uniformly from
$[-r_0,r_0]^2$ and is redrawn if it is not in the interior of the free space. For our experiments, we set $r_0 = 20$. We then vary number of obstacles $m$ from two to seven. For tuning parameters $k = \{20, 40, 60\}$ we run 100 simulations for each $m \in \{2, \dots , 7\}$. Each simulation is terminated successfully when the norm of the difference of $x_t$ and $x^*$ is less than the step size $\eta = 0.01$. A simulation is terminated unsuccessfully if the agent collides with an obstacle - including the outer boundary - or the number of steps reaches $5\times 10^4$.

Figure \ref{m_trend} (a) shows the the results of the simulation for the uncorrected dynamics. For all values of $k$, the ratio of successful trajectories decreases as the number of obstacles increases. This is due to the fact that the increased number of obstacles increases the probability that there an obstacle violates condition \eqref{equ:condition}. In contrast, Figures \ref{m_trend} (b) shows that as $k$ increases the ratio of successful trials increases. For $k = 40$, the success percentage is always above $85\%$. For $k = 60$, the success percentage is always above $95\%$. %Again, we reiterate that the dynamics proposed achieve the same success ratios using less information than the second order corrected dynamics. 

The poor performance of $k = 20$ in the corrected dynamics is due to the fact that we do not consider the outer obstacle $\beta_0$ in the dynamics -- see Section \ref{beta_0lemma}. Because $g_\textrm{nav}$ includes $\beta_0$ as part of the dynamics, the agent is repelled away from the boundary thereby avoiding collision. In contrast, the correction dynamics avoid this collision by assuming that $k$ is large enough such that the agent is always moving inward when it is close to the outer boundary. As expected, the performance improves significantly with larger values of $k$. %Consistent with section \ref{sec_traj_plots}, the similar trends between the dynamics with and without second order correction is due to the fact that the dynamics \eqref{old_dynamics} are almost identical to \eqref{our_dynamics}.

% Discussion can be added later.. Not necessary for numerical results
%\subsection{A Configuration with a Cycle}
%\begin{figure}[!t]
%\centering
%\begin{tabular}{cc}
%		\includegraphics[width = 1.5 in]{plot8.pdf} &
%		\includegraphics[width = 1.5in]{plot7.pdf}\\
%\small (a) & \small (b)
%\end{tabular}
%%\includegraphics[width=.5\linewidth]{plot7.pdf}
%\caption{Trajectories generated in a world whose graph contains a cycle for which Theorem \ref{thm:main} does not hold. (a) All trajectories fail given the uncorrected flow. (b) All trajectories converge to the target despite the presence of a cycle. In this example, we have $k = 10$.}
%\label{cycle_new}
%\end{figure}
%
%Theorem \ref{thm:main} holds only for the case when the configuration is such that there are no cycles defined on its corresponding graph \eqref{equ:graph_def}. An example of this configuration is shown in Figure \ref{cycle_new}. The cycle runs counterclockwise, meaning the north obstacle leads to the east, the east obstacle leads to the south, and so on. This suggests that Theorem \ref{thm:main} can be further generalized to hold without the restriction to acyclic graphs.
%

\vspace{-10pt}
\section{Conclusions} \label{conclusion}
\vspace{-5pt}
We considered the problem of a point agent navigating to a target with a finite number of ellipsoidal obstacles of arbitrary eccentricity. In particular, we directly proposed dynamics which guarantee asymptotic convergence to the target \blue{from almost every initial condition given mild conditions on the target}. We corroborated our theoretical results with numerical simulations on worlds in $\mathbb{R}^2$. 
% \vspace{-3pt}

There are a number of possible extensions to this work. Apart from the generalizations to the case of disc robots, nonstationary obstacles, multiple agents, and online tuning of the parameter $k$, the ellipsoidal condition itself may be lifted.  Given that the exact form of our proposed dynamics relies on estimating the distance and direction to both the target and obstacle, this approach may be used to extend convergence guarantees to not only convex obstacles, but non convex star obstacles as well. In fact, the explicit ellipsoidal form of the obstacles was used in the calculation of the Jacobian of the dynamics evaluated at a critical point only. Therefore, these results can be extended to convex obstacles by considering a general version of these proofs on the gradients of the \blue{obstacle-defining functions} instead of the $x-x_i$ term. 

%Another important observation about the dynamics is that the $x_i$ ``center" term does not necessarily need to be the exact center of an obstacle as prescribed by the quadratic equation. In the case of convex obstacles, any point could be considered a center point, as the term $x_i$ acts more as a pivot point, which suggests the possibility of estimating these points  By the definition of convex obstacles,  \red{Expand: Convex allows agent to only need to estimate any $x_i$ as a center point. }
In a related work, we have shown empirically that the same approach can be used for star obstacles \citep{kumar2020navigation}.  Extending the convergence results to the case of star obstacles requires \blue{generalizing the proof of Theorem \ref{thm:main_result}. Specifically, Lemmas \ref{lem:voronoi} and \ref{lem:invariant} no longer hold due to the lack of separating hyperplanes between star obstacles. Borders between adjacent cells are not necessarily hyperplanes. Instead, Corrolary \ref{cor:hyperplane} would need to be used to describe a finite number of transitions between the two adjacent cells. This immediately breaks the directed acyclic graph on the obstacles. Instead, obstacles should be considered in pairs where the hyperplane does not exist between them, similar to the procedure described in Remark \ref{rem:adjacent}. }
%
%Instead, Corollary \ref{cor:hyperplane} could be used to define transitions between 
%one significant change to the proof, namely an ordering on the obstacles must be established. In this work, we used the Voronoi cells of an obstacle to define adjacent obstacles which induced an ordering by Lemma \ref{lem:voronoi}. This approach will not hold for star obstacles, where we can no longer rely on separating hyperplanes. 
Such a direction would close the gap between navigation with global information by allowing star worlds to be navigated without the need of a diffeomorphism.

%\section{Introduction}
%Video, patres conscripti, in me omnium vestrum ora atque oculos esse 
%conversos, video vos non solunn de vestro ac rei publicae, verum 

%\begin{figure}
%\begin{center}
%\includegraphics[height=4cm]{jcaesar.eps}    % The printed column  
%\caption{Gaius Julius Caesar, 100--44 B.C.}  % width is 8.4 cm.
%\label{fig1}                                 % Size the figures 
%\end{center}                                 % accordingly.
%\end{figure}

% OR

%\begin{figure}
%\begin{center}
%\epsfig{file=jcaesar,width=7cm}
%\caption{Gaius Julius Caesar, 100--44 B.C.}
%\label{fig1}
%\end{center}
%\end{figure}

%\bibliographystyle{model1a-num-names}        % Include this if you use bibtex 
\bibliography{bib-Nav_fun}           % and a bib file to produce the 

                                 % bibliography (preferred). The
                                 % correct style is generated by
                                 % Elsevier at the time of printing.

%\begin{thebibliography}{99}     % Otherwise use the  
                                 % thebibliography environment.
                                 % Insert the full references here.
                                 % See a recent issue of Automatica 
                                 % for the style.
%  \bibitem[Heritage, 1992]{Heritage:92}
%     (1992) {\it The American Heritage. 
%     Dictionary of the American Language.}
%     Houghton Mifflin Company.
%  \bibitem[Able, 1956]{Abl:56}
%     B.~C.~Able (1956). Nucleic acid content of macroscope. 
%     {\it Nature 2}, 7--9. 
%  \bibitem[Able {\em et al.}, 1954]{AbTaRu:54}   
%     B.~C. Able, R.~A. Tagg, and M.~Rush (1954).
%     Enzyme-catalyzed cellular transanimations.
%     In A.~F.~Round, editor, 
%     {\it Advances in Enzymology Vol. 2} (125--247). 
%     New York, Academic Press.
%  \bibitem[R.~Keohane, 1958]{Keo:58}
%     R.~Keohane (1958).
%     {\it Power and Interdependence: 
%     World Politics in Transition.}
%     Boston, Little, Brown \& Co.
%  \bibitem[Powers, 1985]{Pow:85}
%     T.~Powers (1985).
%     Is there a way out?
%     {\it Harpers, June 1985}, 35--47.

%\end{thebibliography}

                                        % in the appendices.
\appendix
\section{Proof of Lemma \ref{lem:global_lyap}} \label{appendix:global_lyap}

%\blue{Since the obstacles $\Oc_i$ are compact and disjoint, and $x^*\notin \Oc_i$, there exists $\varepsilon_1$ such that $N_\varepsilon(\Oc_1), \dots , N_\varepsilon(\Oc_m)$ are pairwise disjoint and $x^* \in F_{>\varepsilon}$ for all $\varepsilon < \varepsilon_1$. }

We have
\begin{equation}
\dot V = \starx^\top \dot x.
\end{equation}
Substituting \eqref{our_dynamics} for $\dot x$, 
\begin{equation} \label{equ:dot_v}
\dot V =-\beta \|x-x^*\|^2 - \frac{\fz}{k}\sum_{j = 1}^m \bar \beta_j(x-x^*)^\top (x-x_j).
\end{equation}
First consider the set $\mathcal{F}_{>\varepsilon}\setminus N_{\varepsilon}(x^*)$.
%\begin{equation}
%    \mathcal{F} \backslash \left(N_\varepsilon (x^*) \cup \bigcup_i N_\varepsilon (\mathcal{O}_i)\right)
%\end{equation}
%
Since we are bounded away from $x-x^*$, we can factor it out of \eqref{equ:dot_v}. Invoking the bounds of \eqref{equ:betai_fz_bound} and \eqref{equ:betai_lower_epsilon}, it holds that $\dot V < 0$ when $K > B\lambda_\varepsilon m/2$. 
%We know that that $\|x-x^*\|^2 > \varepsilon^2 > 0$. Furthermore, since $x\in \F_{>\varepsilon}$, we can say there exists a $\lambda$ such that $\beta_i > \lambda>0$ for all $i$. Since the workspace $\mathcal{X}$ is closed and compact, we can upperbound bound $\beta_j^{-1}$ by $\lambda_\varepsilon^{-1} > 0$. By invoking \eqref{equ:betai_fz_bound}, and \eqref{equ:distance_bound} it follows that for all $k >  B\lambda_\varepsilon^{-1}\varepsilon^{-2}$, $\dot V < 0$.
%
Next, consider evaluating $\dot V$ on $S_i \cup N_\varepsilon (\ccalO_i)$ for any $i$. Split the summation term of \eqref{equ:dot_v} into 
\begin{equation}
\begin{split}
 \dot V &= -\beta \|x-x^*\|^2 + \frac{\fz}{k}\bar \beta_i (x-x^*)^\top (x-x_i) \\
 & + \frac{\fz}{k} \sum_{j \neq i} \bar \beta_j (x-x^*)^\top (x-x_j)\\
 &\leq  -\beta \|x-x^*\|^2 + \frac{\fz}{k}\bar \beta_i \|x-x^*\|\cdot \|x-x_i\|\\
 & + \frac{\fz}{k} \sum_{j \neq i} \bar \beta_j (x-x^*)^\top (x-x_j)\\
 \end{split},
\end{equation}
where the inequality comes from applying Cauchy Schwartz. It holds that for $K> B\lambda_\varepsilon (m-1)/2 - \beta_i^{-1} \|x-x^*\|\cdot \|x-x_i\|$, $\dot V <0$. This is trivially satisfied when $K  > B\lambda_\varepsilon m/2$, as established earlier.

%The first term of \eqref{equ:dot_v} is strictly negative, so we focus on the second terms after factoring out $\beta\fz/k$. In particular, we consider $\beta_i^{-1} (x-x^*)^\top (x-x_j) + \sum_{j = 1, j \neq i}^m \bar \beta_j^{-1} (x-x^*)^\top (x-x_j)$.
%
%By definition of $S_i$, and by invoking \eqref{equ:distance_bound}, \eqref{equ:betai_fz_bound}, and  \eqref{equ:betai_lower_epsilon}, we have that for all $\varepsilon\in (0, \lambda_\varepsilon \gamma B^{-1}m^{-1}=:\varepsilon_1)$, $\dot V < 0$. %know the first term is strictly negative and bounded away from zero, while the second term has a common factor $\beta_i$. There exists an $\varepsilon_{\mathcal{O}_i}$ such that \eqref{equ:si_region} is strictly negative.

Finally, consider $N_\varepsilon (x^*)$. By Cauchy Schwartz and the bounds \eqref{equ:distance_bound} and \eqref{equ:betai_fz_bound}, we bound
\begin{equation}
    \dot V \leq -\beta\|x-x^*\|^2 + \frac{\|x-x^*\|^3}{2k}mB\Lambda_0.
\end{equation}
%
%Because $\fz(x)$ is quadratic, the second term is of order $\left\|x-x^*\right\|^3 =\varepsilon^{3}$ while the first term is only of order $\varepsilon^2$. By invoking \eqref{equ:distance_bound} and \eqref{equ:betai_lower_epsilon}, when $\varepsilon \in (0,\lambda_\varepsilon m^{-1}B^{-1}:=\varepsilon_0)$, $\dot V$ is negative on $N_\varepsilon(x^*)$. %Further, we know that $x^*$ lies in the interior of the free space. Therefore, there exists a $\varepsilon_1 > 0$ such the $\varepsilon$ neighborhoods of the obstacles and target do not intersect.
%
$\dot V$ is negative when $\|x-x^*\| < \lambda_\varepsilon m^{-1} B^{-1}\Lambda_0^{-1}$. 
%Select $\varepsilon < \min\{\varepsilon_0,\varepsilon_1\}$ and $k > K_\varepsilon:= B\lambda_\varepsilon^{-1}\varepsilon^{-2}$ to complete the proof.
This concludes the proof.

%%%%%%%%%%%%%%%%%%%%%%%%%%%%%%%%%%%%%%%%%%%%%%%%%%%%%%%%%%%%%%%%%%%%%%%%%%%%%%%%
%%%   S   E   C   T   I   O   N   %%%%%%%%%%%%%%%%%%%%%%%%%%%%%%%%%%%%%%%%%%%%%%
%%%%%%%%%%%%%%%%%%%%%%%%%%%%%%%%%%%%%%%%%%%%%%%%%%%%%%%%%%%%%%%%%%%%%%%%%%%%%%%%
\section{Proof of Lemma \ref{lem:local_lyap}} \label{appendix:local_lyap}
Evaluate 
%
%\red{Not sure how to fit this on one line as asked by the reviewer}
\begin{align*}
&\dot V_i =\\
&\frac{(x_i-x^*)^\top \dot x}{\|x-x_i\|\cdot \|x_i-x^*\|} - \frac{(x-x_i)^\top(x_i -x^*)(x-x_i)^\top \dot x}{\|x-x_i\|^3\cdot \|x_i-x^*\|}
\end{align*}
Because $x, x^* \in \mathcal{F}$ and $x_i \notin \mathcal{F}$,  $\|x-x_i\|\cdot \|x_i-x^*\|$ and $\beta_j$ for all $j$ are strictly positive. Replace $\dot x$ with \eqref{our_dynamics} to obtain \blue{
\begin{align*}
&\dot V_i =\\
&\frac{\beta}{\|x-x_i\|\cdot \|x_i-x^*\|} \Bigg(-(x_i-x^*)^\top (x-x^*) \\
&+ \frac{(x-x_i)^\top(x_i -x^*)(x-x_i)^\top (x-x^*)}{\|x-x_i\|^2}\\
&+\frac{\fz}{k} \sum_{j =1}^m\frac{1}{\beta_j} \bigg( (x_i-x^*)^\top(x-x_j) \\
&-\frac{(x-x_i)^\top(x_i -x^*)(x-x_i)^\top (x-x_j)}{\|x-x_i\|^2} \bigg)\Bigg)
\end{align*}

Using the fact that $x-x^* = (x-x^i) + (x_i-x^*)$ and the decomposition \eqref{equ:decomp}, the first two terms are strictly negative. Namely, we obtain the bound
\begin{equation}\label{equ:first_two_terms}-\frac{\|x_i-x^*\|^2\cdot \|\rhop\|^2}{a^2 \|x_i-x^*\|^2 + \|\rhop\|^2} \leq -\frac{B\delta_\perp^2}{a^2 B +\delta_\perp^2} < 0,
\end{equation}
where $B$ comes from \eqref{equ:distance_bound}, and we can invoke $\|\textbf{p}\| \geq \delta_\perp$ by the fact that we are in $\ccalC_i\backslash \ccalP_i$.
The terms in the summation where $j = i$ cancel each other out. The summation terms can be upper bounded by invoking \eqref{equ:distance_bound}, \eqref{equ:betai_fz_bound} \eqref{equ:xi_lower_bound}, \eqref{equ:bj_lower_bound}, namely
\begin{equation}
    \fz\sum_{j \neq i}\big(\cdot\big) \leq 2P_0\lambda_{i,\perp}B\cdot\left(1+\frac{1}{\theta_i}\right).
\end{equation}
Selecting 
\begin{equation}\label{equ:Ki}
K_i^\perp:= 2P_0\lambda_{i,\perp}\delta_\perp^{-2}\left(1+ \frac{1}{\theta_i}\right)\left(a^2B+\delta_\perp^2\right) \end{equation} completes the proof.
}

%%%%%%%%%%%%%%%%%%%%%%%%%%%%%%%%%%%%%%%%%%%%%%%%%%%%%%%%%%%%%%%%%%%%%%%%%%%%%%%%
%%%   S   E   C   T   I   O   N   %%%%%%%%%%%%%%%%%%%%%%%%%%%%%%%%%%%%%%%%%%%%%%
%%%%%%%%%%%%%%%%%%%%%%%%%%%%%%%%%%%%%%%%%%%%%%%%%%%%%%%%%%%%%%%%%%%%%%%%%%%%%%%%
\section{Proof of Lemma \ref{lem:unstable}} \label{appendix:unstable}
By $x\in N_\varepsilon(\Oc_i)$, we have that $\beta_i \leq \varepsilon$.% where For the second case, we compare the dynamics of the multi-obstacle system to that of only one obstacle \eqref{one_dynamics_simplified}. 

In the case of one obstacle, there is a critical point where %
\begin{equation} \label{equ:relation}
    \beta_i(x-x^*) = \frac{\fz}{k}(x-x_i).
\end{equation}
This is when $x-x^*$ is aligned with $x-x_i$, when the obstacle is between the agent and the target. 
Let $x_s$ be the point where \eqref{equ:relation} holds. Then we can write for $a > 1$ 
\begin{equation} \label{equ:aligned_a}
    x_s-x^* = a(x_s-x_i).
\end{equation}
This with \eqref{equ:relation} gives the following criterion for the critical point 
\begin{equation} \label{equ:criterion}
    \beta_i(x_s) = \frac{\fz(x_s)}{a k}.
\end{equation}
To show that the point $x_s$ is an unstable equilibrium, consider the Jacobian of \eqref{our_dynamics} $J(g_\textrm{new})$, %evaluated at $x_s$, where the substitution in \eqref{equ:criterion} is made. All functions are evaluated at $x_s$, though omitted for brevity in notation
\begin{equation}
\begin{split}
J(x_s) &= -\beta I - \nabla\beta (x - x^*)^\top  \\
& + \frac{1}{k}\nabla\left(\fz \bar \beta_i\right) (x-x_i)^\top +\frac{1}{k}\fz \bar\beta_i I\\
& + \frac{1}{k} \beta_i J\left(\fz \sum_{j \neq i} \prod_{\ell \neq j} \beta_\ell (x-x_\ell)\right)\\
&+ \frac{1}{k} \nabla \beta_i\left(\fz \sum_{j \neq i} (x-x_j) \right)^\top 
\end{split}
\end{equation}
Define $\mathcal{E}$ to be the set of normal vectors that are orthogonal to $(x_s-x_i)$
\begin{equation}
    \mathcal{E}:= \left\{v\in \mathbb{R}^n \left\vert \|v\| = 1,  v^\top (x_s-x_i) = 0\right.\right\}.
\end{equation}
The rank of $\mathcal{E}$ is $n-1$. Let $\left\{v_1, \dots, v_{n-1}\right\}$ be a basis. Consider any $v_i$ and evaluate the Jacobian at $x_s$ where \eqref{equ:relation} holds, 
\begin{equation}
\begin{split}
     v_i^\top &J(g_\textrm{new})v_i \big|_{x_s} = -\beta +\frac{1}{k}\fz \bar\beta_i \\
    &+\frac{1}{k}v_i^\top \beta_i J\left(\fz \sum_{j \neq i} \left(\prod_{\ell \neq j} \beta_\ell\right) (x-x_j)\right)v_i \\
    &= -\frac{\bar \beta_i \fz}{ak} +\frac{1}{k}\fz \bar\beta_i \\
    &+ \frac{1}{ak^2}v_i^\top  J\left(\fz^2 \sum_{j \neq i} \left(\prod_{\ell \neq j} \beta_\ell\right) (x-x_j)\right)v_i 
    %& + \frac{1}{k^2} v_i^\top J \left(\fz^2 \sum_{j \neq i }\prod_{\ell \neq j} \beta_\ell (x_s - x_\ell)\right) v_i
\end{split}
\end{equation}
We because of \eqref{equ:betai_fz_bound} and  \eqref{equ:distance_bound}, for any unit vector $n$, we can bound the final term as 
\begin{equation}
    \left\|n^\top J\left( \fz^2\sum_{j \neq i} \prod_{\ell \neq j} \beta_\ell (x-x_\ell) \right) n \right\| \leq C_1
\end{equation}
Let $p_0:= \min_{x\notin C_0} \fz(x)$ and recall \eqref{equ:betai_lower_adjacent}. %Additionally, by \eqref{equ:distance_bound} and \eqref{equ:betai_fz_bound}
%we can bound the final term by some $C_1$. 
As such, for $$K_{\rhop, i_1} = \frac{C_1}{p\varepsilon^{m-1}} \cdot \frac{a}{a-1},$$
with any $k > K_{\rhop, i_1}$, we know that $v_i^\top J(x_s)v_i$ is positive rendering that direction unstable. 

Now consider the unit vector $w$ aligned with $(x_s-x_i)$. In particular, we set $w = (x_s-x_i)/\|x_s-x_i\|$. 
Consider evaluating  $w^\top J(g_\textrm{new})w$ at $x_s$,
\begin{equation}
    \begin{split}
        w^\top &J(g_\textrm{new})w \bigg|_{x_s} = -\beta - w^\top \nabla \beta(x-x^*)^\top w + \frac{1}{k}\fz \bar \beta_i \\ 
        &+ \frac{1}{k}w^\top \nabla \left(\fz \bar\beta_i\right)(x-x_i)^\top w \\
        & + \frac{\beta_i}{k} w^\top J\left(\fz \sum_{j \neq i} \prod_{\ell \neq j} \beta_\ell \right)w\\
        & + \frac{1}{k}w^\top \nabla \beta_i \left(\fz\sum_{j \neq i}\left(\prod_{\ell \neq j}\beta_\ell\right)(x-x_j) \right)w
    \end{split}
\end{equation}
Split the second term using \eqref{equ:nabbeta}, and apply the relations \eqref{equ:relation} and \eqref{equ:aligned_a} to obtain 
\begin{equation}
    \begin{split}
         w^\top &J(g_\textrm{new})w \bigg|_{x_s} = -a^{-1}\bar\beta_i \\
        &+\frac{1}{k} \Bigg( \fz\bar\beta_i\left(1 - a^{-1}\right) \\
        &- \fz a^{-1}w^\top \big(\sum_{j \neq i} \prod_{\ell \neq j} \beta_\ell \nabla \beta_\ell \big)(x-x^*)^\top w\\
        &+\fz a^{-1} \sum_{j \neq i } (x-x_j)^\top w\\
        &+ w^\top \nabla (\fz \bar \beta_i)(x-x_i)^\top w \Bigg)\\
        &+ \frac{1}{k^2} \Bigg( \fz a^{-1}w^\top J\big(\fz \sum_{j \neq i}\prod_{\ell \neq j} \beta_\ell (x-x_j)\big)w\Bigg) \\
        & \leq -\bar \beta_i a^{-1} + C_2k^{-1} + C_3 k^{-2},
    \end{split}
\end{equation}
where 
\begin{equation}
    \begin{split}
        &\Bigg\|\fz\bar\beta_i\left(1 - a^{-1}\right) +\fz a^{-1} \sum_{j \neq i } (x-x_j)^\top w\\
        &- \fz a^{-1}w^\top \big(\sum_{j \neq i} \prod_{\ell \neq j} \beta_\ell \nabla \beta_\ell \big)(x-x^*)^\top w\\
        &+ w^\top \nabla (\fz \bar \beta_i)(x-x_i)^\top w \Bigg\| \leq C_2,
    \end{split}
\end{equation}
and 
\begin{equation}
    \Bigg\| \fz a^{-1}w^\top J\big(\fz \sum_{j \neq i}\prod_{\ell \neq j} \beta_\ell (x-x_j)\big)w\Bigg\| \leq C_3.
\end{equation}
$C_2$ and $C_3$ are finite by \eqref{equ:distance_bound} and \eqref{equ:betai_fz_bound}.
Then, for 
$$K_{\rhop ,i_2} := \max\left\{ \frac{C_2\pm \sqrt{C_1^2+4C_3 \bar\beta_i(x_{s,i})a^{-1}} }{a} \bar \beta_i(x_{s,i})\right\},$$
with any $k> K_{\rhop ,i_2}$, we know that $w^\top J(g_\textrm{new})w$ is negative.
%Given that this holds for all basis vectors of $\mathcal{E}$, $x_s$ is an unstable equilibrium. 
Choose $K_{\rhop , i}= \max\{K_{\rhop, i_1}, K{\rhop, i_2}\}$ to complete the proof.

\end{document}